\documentclass[11pt]{article}

\usepackage[T1]{fontenc}
\usepackage[utf8]{inputenc}
\usepackage{lmodern}
\usepackage[margin=1in]{geometry}
\usepackage{amsmath,amssymb,amsthm,mathtools,bm}
\usepackage{microtype}
\usepackage{xcolor}
\usepackage{graphicx}
\usepackage[pdfpagelabels,breaklinks]{hyperref}

\allowdisplaybreaks[2]
\numberwithin{equation}{section}
\hypersetup{
  colorlinks=true,
  linkcolor=blue!45!black,
  citecolor=green!35!black,
  urlcolor=blue!55!black,
  pdftitle={On the Nonoccurrence of Gradient Restart in Nesterov Accelerated Gradient Method},
  pdfauthor={Chenglong Bao, Liang Chen, Jiahong Li, Defeng Sun}
}

\newtheorem{theorem}{Theorem}[section]
\newtheorem{lemma}[theorem]{Lemma}

\newtheorem{corollary}[theorem]{Corollary}

\theoremstyle{remark}
\newtheorem{remark}[theorem]{Remark}

\newcommand{\R}{\mathbb R}
\newcommand{\Z}{\mathbb Z}
\newcommand{\ip}[2]{\left\langle #1,#2\right\rangle}
\newcommand{\norm}[1]{\left\lVert #1\right\rVert}
\newcommand{\dd}{\,\mathrm d}
\newcommand{\ee}{\mathrm e}
\newcommand{\fstar}{f^\star}

\title{\Large\bf On the Nonoccurrence of Gradient Restart in Nesterov Accelerated Gradient Method\footnote{All authors contributed equally to this work.
The research of this work was supported by the National Natural Science Foundation of China (12625119, W2612010, 12671361), the Hong Kong RGC Senior Research Fellow Scheme (SRFS2223-5S02), and GRF Project (15307523, 15305324, 15309625).}}
\author{ Chenglong Bao\thanks{Yau Mathematical Sciences Center, Tsinghua University, Beijing, China, and Yanqi Lake Beijing Institute of Mathematical Sciences and Applications, Beijing, China (\url{clbao@tsinghua.edu.cn}).}
\and
Liang Chen\thanks{School of Mathematics, Hunan University, Changsha, China (\url{chl@hnu.edu.cn}).}
\and
Jiahong Li\thanks{Department of Applied Mathematics, The Hong Kong Polytechnic University, Hung Hom, Hong Kong (\url{jiahong24.li@polyu.edu.hk}).}
\and
Defeng Sun\thanks{Department of Applied Mathematics, The Hong Kong Polytechnic University, Hung Hom, Hong Kong (\url{defeng.sun@polyu.edu.hk}).}
}
\date{\today}

\begin{document}
\maketitle

\begin{abstract}
Gradient restarting often improves the practical performance of accelerated gradient methods for unconstrained convex optimization problems, but a general theoretical guarantee of such improvement remains unclear.
In this paper, we construct smooth, strongly convex objectives with Lipschitz-continuous gradients and corresponding initial points for which gradient restart never occurs, both for Nesterov's accelerated gradient method and for its ordinary differential equation model.
From these examples, we conclude that strong convexity and smoothness alone do not guarantee that gradient restart improves the convergence behavior of the accelerated method or its continuous-time model.
In particular, for the continuous-time example, the objective gap is asymptotic to $1/(\pi t^3)$.
We therefore obtain a negative answer to the open problem raised by Su, Boyd, and Cand{\`e}s [\textit{J. Mach. Learn. Res.}, 2016, 17(153), 1--43] of whether gradient restart guarantees linear convergence of the continuous dynamics for smooth strongly convex objectives.
\end{abstract}

\medskip
{\small
\begin{center}
\parbox{0.95\hsize}{{\bf Keywords.}\; Nesterov's accelerated gradient method, gradient restarting, convex optimization, counterexample, R-linear convergence}
\end{center}
\begin{center}
\parbox{0.95\hsize}{{\bf MSC Classification.}\; 90C25, 65K05, 65B05, 90C06, 90C30}
\end{center}}

\section{Introduction}
We consider the unconstrained optimization problem
\begin{equation}
\min_{x\in\R^n} f(x),
\label{eq:optimization-problem}
\end{equation}
where $f:\R^n\to\R$ is convex and continuously differentiable with $L$-Lipschitz-continuous gradient for some $L>0$.
We assume that $f$ attains its minimum and write $f^\star:=\min_{x\in\R^n}f(x)$.
For solving \eqref{eq:optimization-problem}, we consider Nesterov's accelerated gradient method (NAG)~\cite{nesterov1983,nesterov2004} in the form, for $k=0,1,\ldots$,
\begin{equation}
x_{k+1}:=y_k-s\nabla f(y_k)
\qquad\mbox{and}\qquad
y_{k+1}:=x_{k+1}+\frac{k}{k+3}(x_{k+1}-x_k)
\label{eq:discrete-nag}
\end{equation}
where the initial point $x_0\in\R^n$ and $y_0:=x_0$ are given, and the step size $s\in(0,1/L]$ is fixed.
NAG \eqref{eq:discrete-nag} improves the worst-case objective rate of gradient descent from $O(k^{-1})$ to $O(k^{-2})$, which is optimal up to constants among first-order methods for smooth convex optimization \cite{nesterov1983,nesterov2004}.
Nevertheless, due to extrapolation, NAG often exhibits oscillatory behavior, which can negatively impact practical convergence performance.
To address this problem, various restart techniques have been proposed \cite{odonoghue2015,tao2016local,liang2017activity}.
Among them, a widely used strategy is gradient restart \cite{odonoghue2015}, which resets the momentum when
\begin{equation}
D_k:=\ip{\nabla f(y_k)}{x_{k+1}-x_k}>0,
\label{eq:discrete-gr}
\end{equation}
and uses $x_k$ as the new initial point.
For fixed $s$, $f$, and a nonoptimal initial point $x_0$, define the first gradient-restart index by
\begin{equation}
  K^{\rm gr}(x_0;f)
  :=\sup\left\{m\in\Z_{++}\mid
  D_k\le0,\ \forall k\in\Z\cap[0,m-1]\right\}.
  \label{gr ft}
\end{equation}
Evaluating $D_k$ requires only one additional inner product.
This inexpensive criterion has proved effective in applications including coherence retrieval, quantum tomography, and classification \cite{bao2018coherence,shang2017superfast,ito2017unified}.

These empirical improvements over the unrestarted NAG motivate a theoretical analysis of when gradient restart can improve convergence.
For smooth strongly convex objectives and $0<s<1/L$, Bao et al.~\cite{bao2026} established global $R$-linear convergence for NAG and for the gradient-restarted variant.
Bao et al. further established an improved convergence estimate of the gradient restart scheme when successive restart intervals are uniformly bounded, i.e., $K^{\rm gr}(x_i;f)\le \overline{\cal K}$ for some $\overline{\cal K}\ge 2$ and all $i\in \mathcal{I}$ with $\mathcal{I}$ denoting all restart and initial indices~\cite[Proposition~2]{bao2026}.
This theoretical improvement prompts further investigation into whether smoothness and strong convexity suffice to ensure this additional boundedness condition, or more fundamentally, whether they guarantee the necessary condition $K^{\rm gr}(x_0;f)<+\infty$.
Moreover, if $K^{\rm gr}(x_0;f)=+\infty$, the restart criterion is never activated during the iteration, which means the gradient restart scheme coincides with the unrestarted NAG, thereby ruling out any convergence improvement.
We therefore ask the following question:
\begin{quote}
\noindent
\bf
Q1: Do there exist an $L$-smooth, strongly convex function $f$, a step size $0<s<1/L$, and a nonoptimal initial point $x_0$ such that $K^{\rm gr}(x_0;f)=+\infty$?
\end{quote}

The effect of gradient restart on convergence can also be studied through the continuous-time model of NAG.
As shown in \cite{su2016}, under the scaling $t=k\sqrt{s}$, the iteration \eqref{eq:discrete-nag} admits the following continuous-time limit as $s\downarrow0$:
\begin{equation}
\ddot X(t)+\frac{3}{t}\dot X(t)+\nabla f(X(t))=0,
\qquad t>0,
\qquad
X(0)=x_0,
\qquad\mbox{and}\qquad
\dot X(0)=0.
\label{eq:nag-ode}
\end{equation}
For convex $f$ with an $L$-Lipschitz-continuous gradient, the NAG-ODE \eqref{eq:nag-ode} has a unique global solution $X\in C^2((0,+\infty);\R^n)\cap C^1([0,+\infty);\R^n)$ for all $x_0\in\R^n$ \cite[Theorem~1]{su2016}.
From \cite[Theorem~3]{su2016}, we obtain the accelerated objective convergence rate $f(X(t))-f^\star=O(t^{-2})$ for the NAG-ODE \eqref{eq:nag-ode}.
We can also use this model to study NAG's oscillatory behavior.
For strongly convex quadratic objectives and nonoptimal initial points, the objective gap satisfies the sharp bound $f(X(t))-f^\star=O(t^{-3})$ and does not decay exponentially \cite[Section~3.2]{su2016}, in contrast to the $R$-linear convergence of discrete NAG for $0<s<1/L$ \cite{bao2026}.
As in the discrete case, we consider gradient restarting to overcome this limitation by resetting the dynamics whenever
\begin{equation}
G(t):=
\ip{\nabla f(X(t))}{\dot X(t)}=
\frac{\mathrm d}{\mathrm dt}f(X(t))\ge0.
\label{eq:continuous-gr}
\end{equation}
Given a nonoptimal initial point $x_0$, the gradient restart time is then defined by
\begin{equation}
\label{grt}
T^{\rm gr}(x_0;f)
:=\sup\left\{t>0\mid G(u)<0\ \text{for all }u\in(0,t)\right\}.
\end{equation}
The gradient-restarted trajectory $X^{\mathrm{gr}}$ is obtained by solving \eqref{eq:nag-ode} anew at each finite restart time from the current position, with zero velocity and local time reset to zero, and concatenating the resulting solution segments in total elapsed time.
For convex quadratic objectives, Bao et al.~\cite{bao2026} proved that gradient restart improves the sublinear convergence of the unrestarted dynamics to linear convergence.
A key step in their analysis is to prove $T^{\rm gr}(x_0;f)<+\infty$ for every nonoptimal initial point $x_0$.
However, it remains unclear whether this finiteness property can be extended to general smooth strongly convex objectives, which motivates the basic question:
\begin{quote}
\noindent
\bf
Q2: Do there exist an $L$-smooth, strongly convex function $f$ and a nonoptimal initial point $x_0$ such that $T^{\rm gr}(x_0;f)=+\infty$?
\end{quote}
An affirmative answer to \textbf{Q2} would show that smoothness and strong convexity alone cannot guarantee the convergence improvement of gradient restarting, since the restarted and unrestarted trajectories would coincide.
If, in addition, the common restarted and unrestarted trajectory fails to converge linearly, we would obtain a negative answer to the open problem of whether gradient restart guarantees linear convergence of the continuous-time dynamics for $L$-smooth, strongly convex functions~\cite{su2016}.

\subsection*{Main results}
For $L>0$, let $\mathcal F_L(\R^n)$ be the class of $L$-smooth convex functions on $\R^n$:
\begin{equation*}
\begin{aligned}
\mathcal{F}_{L}(\R^n):=\big\{f:\R^n\to\R\mid{}&f\text{ is convex and
differentiable, and}\\
&\norm{\nabla f(x)-\nabla f(y)}\le L\norm{x-y}
\quad\forall x,y\in\R^n\big\}.
\end{aligned}
\end{equation*}
For $\mu>0$, let $\mathcal S_\mu(\R^n)$ be the class of $\mu$-strongly convex functions on $\R^n$:
\begin{equation*}
\mathcal{S}_\mu(\R^n):=\big\{f:\R^n\to(-\infty,+\infty]\mid
f(x)-\frac{\mu}{2}\norm{x}^2\text{ is convex}\big\}.
\end{equation*}
We write $\mathcal{S}_{\mu,L}(\R^n):=\mathcal{F}_L(\R^n)\cap \mathcal{S}_\mu(\R^n)$.
We answer \textbf{Q1} and \textbf{Q2} affirmatively in Theorems~\ref{thm:discrete-counterexample} and~\ref{thm:continuous-counterexample}, respectively.

\begin{theorem}
\label{thm:discrete-counterexample}
There exist constants $0<\mu\le L$, a function $f\in\mathcal{S}_{\mu,L}(\R^2)$, a step size $s\in(0,1/L)$, and an initial point $x_0=y_0\in\R^2$ such that the sequence generated by the NAG iteration \eqref{eq:discrete-nag} satisfies
\begin{equation}
  \ip{\nabla f(y_k)}{x_{k+1}-x_k}<0 \quad \mbox{for all} \quad k\ge0,
  \label{eq:discrete-main-nonrestart}
\end{equation}
and therefore $K^{\rm gr}(x_0;f)=+\infty$, where $K^{\rm gr}(x_0;f)$ is defined in \eqref{gr ft}.
\end{theorem}


\begin{theorem}
\label{thm:continuous-counterexample}
There exist constants $0<\mu\le L$, a function $f\in\mathcal{S}_{\mu,L}(\R^2)$, and an initial point $x_0\in\R^2$ such that the unique solution of the NAG-ODE \eqref{eq:nag-ode} satisfies
\begin{equation}
  \ip{\nabla f(X(t))}{\dot X(t)}<0 \quad \mbox{for all} \quad t>0,
  \label{eq:continuous-main-nonrestart}
\end{equation}
that is, $T^{\rm gr}(x_0;f)=+\infty$, where $T^{\rm gr}(x_0;f)$ is defined in \eqref{grt}.
\end{theorem}

The explicit constructions of the discrete and continuous counterexamples, together with the proofs of the two theorems above, are given in Sections~\ref{sec:discrete} and~\ref{sec:continuous}, respectively.
Since no restart occurs in either construction, the restarted and unrestarted methods generate the same iterates or trajectory.
Moreover, for the continuous-time construction in Theorem~\ref{thm:continuous-counterexample}, we further establish in Corollary~\ref{coro:continuous-counterexample} below that the objective gap is asymptotic to $1/(\pi t^3)$, which rules out linear convergence and thus answers the open problem in \cite[p.27]{su2016} negatively.
\begin{corollary}
\label{coro:continuous-counterexample}
There exist constants $0<\mu\le L$, a function $f\in\mathcal{S}_{\mu,L}(\R^2)$, and an initial point $x_0\in\R^2$ such that the gradient-restarted trajectory $X^{\mathrm{gr}}$ associated with the NAG-ODE \eqref{eq:nag-ode} and the restart criterion \eqref{eq:continuous-gr} satisfies
\begin{equation}
  f(X^{\mathrm{gr}}(t))-\fstar
  =\frac{1}{\pi t^3}+O(t^{-5}) \quad \mbox{as} \quad t\to+\infty.
  \label{eq:continuous-main-properties}
\end{equation}
Here $x^\star$ is the unique minimizer of $f$ with $f^\star=f(x^\star)$.
Thus, no constants $t_0,C,c_0>0$ satisfy $f(X^{\mathrm{gr}}(t))-\fstar\le C\ee^{-c_0t}$ for every $t\ge t_0$.
\end{corollary}

\subsection*{Related work}
Wen et al.~\cite{wen2017linear} established $R$-linear convergence of inertial forward--backward splitting under an error-bound condition, but their requirement $\sup_k\beta_k<1$ excludes the momentum parameters of NAG \eqref{eq:discrete-nag}.
For accelerated methods whose momentum parameters approach one, Tao et al.~\cite{tao2016local} and Liang et al.~\cite{liang2017activity} used spectral analysis to establish local linear convergence and explain oscillatory behavior near a minimizer.
To obtain a nonasymptotic convergence bound, Li et al.~\cite{li2024linear} obtained an $R$-linear rate after a finite iteration threshold depending on $L,\mu$, and $s$, and extended the analysis to the fast iterative shrinkage-thresholding algorithm for composite strongly convex problems.
Nevertheless, these linear convergence results of accelerated methods do not improve on those of gradient or proximal gradient methods and can even be worse near a minimizer, consistent with the oscillatory behavior of accelerated iterates \cite{bao2026,tao2016local,liang2017activity}.

Restarting is often used to improve the convergence performance of iterative methods, as related strategies have been widely studied in conjugate gradient \cite{powell1977}, the generalized minimal residual method \cite{saad1986}, Halpern-type methods \cite{sun2025accelerating,chen2026hprlp}, and Anderson acceleration \cite{zhang2020,wei2022shortterm,wei2022minimal}.
For NAG \eqref{eq:discrete-nag}, a suitably chosen fixed restart period allows one to attain an accelerated linear convergence rate for smooth strongly convex objectives \cite{becker2011templates,Nesterov2013GradientMF,odonoghue2015}.
Varying restart periods and parameter searches provide alternatives \cite{roulet2020sharpness}, while Renegar and Grimmer \cite{renegar2022restart} trigger restarts through prescribed decreases in the objective value.
For schedules based on problem parameters, however, inaccurate estimates or changes in local conditioning can limit their practical benefit \cite{odonoghue2015}.

This dependence motivates adaptive restart rules that use information from the current iterates or trajectory to determine when to restart.
Besides gradient restart, common adaptive rules include function-value restart, triggered by an increase in the objective value \cite{odonoghue2015}, and speed restart, triggered when the speed ceases to increase \cite{su2016}.
Compared to gradient restart, the former requires additional function-value evaluations, while the latter is more conservative and may only partially mitigate the adverse effects of oscillation.
Su, Boyd and Cand{\`e}s \cite{su2016} proved uniform bounds on speed-restart intervals and thus established global linear convergence under strong convexity.
For inertial systems with Hessian-driven damping \cite{attouch2016hessian,attouch2022hessian}, Maul\'en and Peypouquet \cite{maulen2023speed} established linear convergence of a speed-restart scheme under quadratic growth.
Maul\'en, Guo, and Peypouquet \cite{maulen2026interpolate} further established linear convergence of restart schemes interpolating between speed and function-value restart under strong convexity.
However, along any trajectory of the NAG-ODE \eqref{eq:nag-ode}, the speed-restart condition is always reached before the gradient-restart condition \eqref{eq:continuous-gr}, so upper bounds on speed-restart intervals cannot guarantee that gradient restart is ever triggered.

In several numerical examples, gradient-restarted NAG reaches small objective gaps in fewer iterations than both unrestarted NAG and speed-restarted NAG \cite[Figure~6]{su2016}.
However, the available theoretical guarantees of improvement for gradient restart remain limited.
Specifically, Moursi et al.\ \cite{moursi2023gradient} proved an improved $O(k^{-2})$ function-value bound for a variant restarting from $x_{k+1}$ on one-dimensional smooth convex objectives.
Under strong convexity, Bao et al.\ \cite{bao2026} established global $R$-linear convergence in discrete time, yet their guarantees of improved convergence rates still require uniformly bounded restart intervals in discrete time or quadratic objectives in continuous time.
Thus, all existing improvement guarantees for gradient restart either restrict the problem class or assume bounded restart intervals, rather than following from smoothness and strong convexity alone. 

\subsection*{Organization}
The remainder of the paper is organized as follows.
Section~\ref{sec:preliminaries} reviews preliminary results on the two classes of special functions needed for our constructions.
Sections~\ref{sec:discrete} and~\ref{sec:continuous} construct the counterexamples for discrete NAG and the continuous-time NAG-ODE, and verify the required properties, respectively.
Section~\ref{sec:conclusion} discusses the common principles behind the constructions, their extensions, and further questions.

\section{Preliminaries}
\label{sec:preliminaries}
This section collects several identities and asymptotic estimates for the special functions used in our counterexamples.
Because both constructions build upon solutions for isotropic quadratic objectives, the underlying discrete NAG recurrences and continuous NAG-ODE trajectories can be expressed in terms of Ferrers and Bessel functions, respectively.
We review the relevant properties of Legendre and Ferrers functions~\cite{erdelyi1953,durand1978} in Subsection~\ref{subsec:ferrers-preliminaries}, and those of Bessel functions~\cite{watson1944} in Subsection~\ref{subsec:bessel-preliminaries}.

\subsection{Legendre and Ferrers functions}
\label{subsec:ferrers-preliminaries}

To analyze the discrete NAG recurrence for the quadratic objective $z \mapsto \mu \Vert{}z\Vert{}^2/2$ with $\mu > 0$ and $0 < s < 1/\mu$, we make use of the associated Ferrers functions of order $-1$.
For each integer $m \ge 0$, the Legendre polynomial $P_m$ on $\mathbb{R}$ is given by Rodrigues' formula~\cite{erdelyi1953}, namely
\begin{equation}
  P_m(x)
  :=
  \frac{1}{2^m m!}
  \frac{\mathrm d^m}{\mathrm dx^m}(x^2-1)^m.
  \label{eq:legendre-rodrigues}
\end{equation}
For $x \in (-1, 1)$, the Ferrers function of the second kind~\cite{erdelyi1953} is defined by
\begin{equation}
  Q_m(x):=\frac12P_m(x)\log\frac{1+x}{1-x}
    -\sum_{j=1}^m\frac1jP_{j-1}(x)P_{m-j}(x)
  \label{eq:second-kind-Q-definition}
\end{equation}
for integers $m \ge 1$.
Following the normalization in~\cite{erdelyi1953}, we define the associated Ferrers functions of order $-1$ for integers $m \ge 1$ and $x \in (-1, 1)$ via the derivative relations
\begin{equation}
  P_m^{(-1)}(x):=\frac{\sqrt{1-x^2}}{m(m+1)}P_m'(x),
  \qquad
  Q_m^{(-1)}(x):=\frac{\sqrt{1-x^2}}{m(m+1)}Q_m'(x).
  \label{eq:negative-order-identities}
\end{equation}
These functions satisfy the following three-term recurrence and differentiation identities.

\begin{lemma}[\cite{erdelyi1953}~Section~3.8, Eqs.~(12) and~(19)]
\label{lem:ferrers-identities}
Let $\{W_m\}_{m\ge1}$ denote either $\{P_m^{(-1)}\}_{m\ge1}$ or $\{Q_m^{(-1)}\}_{m\ge1}$.
For every integer $m\ge2$ and $x\in(-1,1)$, we have
\begin{align}
  (m+2)W_{m+1}(x)
  &=(2m+1)xW_m(x)-(m-1)W_{m-1}(x)
  \label{eq:ferrers-degree-recurrence}
  \\
  \mbox{and}\quad (1-x^2)W_m'(x)
  &=(m-1)W_{m-1}(x)-mxW_m(x).
  \label{eq:ferrers-differentiation}
\end{align}
\end{lemma}

We next record the large-degree asymptotics at a fixed angle $\theta$.
\begin{lemma}[\cite{erdelyi1953}~Section~3.9.1, Eqs.~(1) and (2)]
\label{lem:ferrers-asymptotics}
Fix $\theta\in(0,\pi)$ and set $\omega_m:=(m+\tfrac12)\theta-\tfrac{3\pi}{4}$.
Then, as $m\to+\infty$, we have
\begin{align}
  P_m^{(-1)}(\cos\theta)
  &=\sqrt{\frac{2}{\pi\sin\theta}}\,m^{-3/2}
    \bigl(\cos\omega_m+O(m^{-1})\bigr)
  \label{eq:ferrers-P-large}
  \\
  \mbox{and}\quad \frac2\pi  Q_{m}^{(-1)}(\cos\theta)
  &=-\sqrt{\frac{2}{\pi\sin\theta}}\,m^{-3/2}
    \bigl(\sin\omega_m+O(m^{-1})\bigr).
  \label{eq:ferrers-Q-large}
\end{align}
\end{lemma}
\begin{proof}
Specializing the fixed-angle expansions in~\cite[Section~3.9.1, Eqs.~(1) and~(2)]{erdelyi1953} to degree $m$ and order $-1$, and applying the gamma-function ratio estimate $\Gamma(m)/\Gamma(m + 3/2) = m^{-3/2}(1 + O(m^{-1}))$ from~\cite[Section~1.18, Eq.~(4)]{erdelyi1953}, we obtain~\eqref{eq:ferrers-P-large} and~\eqref{eq:ferrers-Q-large}.
\end{proof}

To determine the sign of the discrete restart quantity, we need the following Durand-type formula, which also involves the Legendre function of the second kind on $(1,+\infty)$.
To distinguish it from $Q_m$ on $(-1,1)$, we denote this function by $\widehat Q_m$ and define it, for integers $m\ge0$ and $w>1$, by~\cite{erdelyi1953}
\begin{equation}
  \widehat Q_m(w):=\frac12\int_{-1}^1\frac{P_m(v)}{w-v}\dd v.
  \label{eq:legendre-Q-integral}
\end{equation}

\begin{lemma}[\cite{durand1978}~Section~3.3, Eq.~(3.11)]
\label{lem:ferrers-nicholson}
For every integer $m\ge2$ and $x\in(0,1)$, we have
\begin{equation}
  P_m'(x)^2+\frac4{\pi^2}Q_m'(x)^2
  =-\frac{4m(m+1)}{\pi^2(1-x^2)^2}
    \int_1^{+\infty} \widehat Q_m'(w)
    \sqrt{(w-1)(w+1-2x^2)}\dd w.
  \label{eq:ferrers-derivative-nicholson}
\end{equation}
Moreover, $\widehat{Q}_m'(w)<0$ for all $w>1$ while the integral in~\eqref{eq:ferrers-derivative-nicholson} converges absolutely.
\end{lemma}

\begin{proof}

To establish~\eqref{eq:ferrers-derivative-nicholson}, we specialize Durand's formula~\cite[Section~3.3, Eq.~(3.11)]{durand1978} (which remains unaffected by the erratum~\cite{cohl2026erratum}) to degree $m - 1$, auxiliary index $0$, and parameter $\alpha = 3/2$.
Under the normalization in~\cite[Eqs.~(3.1)--(3.4)]{durand1978}, this specialization reads
\[
P_m'(x)^2 + \frac{4}{\pi^2} Q_m'(x)^2 = -\frac{4m(m+1)}{\pi^2} \int_1^{+\infty} \widehat{Q}_m'\bigl(x^2 + (1 - x^2)t\bigr)\sqrt{t^2 - 1}\,\mathrm{d}t.
\]
Then the identity \eqref{eq:ferrers-derivative-nicholson} is immediately derived by substitution $w = x^2 + (1 - x^2)t$.

Integrating \eqref{eq:legendre-Q-integral} by parts $m$ times via Rodrigues' formula \eqref{eq:legendre-rodrigues} and differentiating under the integral sign, we obtain
\begin{equation*}
\widehat{Q}_m'(w) = -\frac{m+1}{2^{m+1}} \int_{-1}^1 \frac{(1-v^2)^m}{(w-v)^{m+2}}\,\mathrm{d}v < 0 \qquad \text{for all } w > 1.
\end{equation*}
Since $w-v\ge w-1$ for $v\in[-1,1]$, we obtain $\widehat Q_m'(w)=O(w^{-m-2})$ as $w\to+\infty$ from the preceding representation.
Moreover, it follows from the definition~\eqref{eq:legendre-Q-integral} that
\begin{equation*}
  |\widehat Q_m'(w)|
  \le
  \frac{\max_{v\in[-1,1]}|P_m(v)|}{w^2-1},
  \qquad w>1.
\end{equation*}
Hence the integrand in \eqref{eq:ferrers-derivative-nicholson} is $O((w-1)^{-1/2})$ as $w\downarrow1$ and $O(w^{-m-1})$ as $w\to+\infty$.
We thus obtain absolute convergence, completing the proof.
\end{proof}

\subsection{Bessel functions}
\label{subsec:bessel-preliminaries}

For the one-dimensional quadratic objective $f(z) = z^2/2$, the substitution $w(t) = tz(t)$ transforms the NAG-ODE~\eqref{eq:nag-ode} into Bessel's equation of order one~\cite{watson1944}:
\begin{equation}
\label{eq:bessel-equation}
t^2 w''(t) + tw'(t) + (t^2 - 1)w(t) = 0, \qquad \text{for } t > 0.
\end{equation}
The Bessel functions $J_1$ and $Y_1$ of the first and second kinds form a fundamental system of solutions of~\eqref{eq:bessel-equation}~\cite[Section~3.63, Eq.~(1)]{watson1944}.
More generally, for $t > 0$ and real order $\nu$, the Bessel function of the first kind is defined by the convergent series~\cite[Section~3.1, Eq.~(8)]{watson1944}
\begin{equation}
\label{eq:J-definition}
J_\nu(t) := \sum_{j=0}^{+\infty} \frac{(-1)^j}{j!\,\Gamma(j + \nu + 1)} \left(\frac{t}{2}\right)^{2j + \nu},
\end{equation}
while for noninteger $\nu$, the Bessel function of the second kind (in Weber's normalization) is defined by the connection formula~\cite[Section~3.54, Eqs.~(1)--(2)]{watson1944}
\[
Y_\nu(t) := \frac{\cos(\pi\nu)J_\nu(t) - J_{-\nu}(t)}{\sin(\pi\nu)}.
\]
For integer orders $\nu = m \in \mathbb{Z}$, since $J_{-m}(t) = (-1)^m J_m(t)$, both the numerator and denominator vanish at $\nu = m$, and $Y_m(t) := \lim_{\nu \to m} Y_\nu(t)$ is well defined by continuity in $\nu$.
We first record the differentiation identity needed below.

\begin{lemma}[\cite{watson1944}~Section~3.2 Eq.~(4) and~Section~3.56, Eq.~(4)]
\label{lem:bessel-differentiation}
Let $\{C_m\}_{m\ge1}$ denote either $\{J_m\}_{m\ge1}$ or $\{Y_m\}_{m\ge1}$.
For every integer $m\ge1$ and $t>0$, we have
\begin{equation}
  \frac{\mathrm d}{\mathrm dt}\bigl(t^{-m}C_m(t)\bigr)
  =-t^{-m}C_{m+1}(t).
  \label{eq:bessel-scaled-derivative}
\end{equation}
\end{lemma}

For estimates near $t=0$, we introduce the following ascending series representations.
\begin{lemma}[\cite{watson1944}~Section~3.1, Eq.~(8) and Section~3.52, Eq.~(3)]
\label{lem:bessel-ascending}
For every integer $m\ge1$ and every $t>0$, the Bessel functions $J_m$ and $Y_m$ admit the convergent series representations
\begin{align}
  J_m(t)
  &=\sum_{j=0}^{+\infty}
    \frac{(-1)^j}{j!(m+j)!}\left(\frac t2\right)^{m+2j}
  \label{eq:Jn-ascending}
  \\
  \mbox{and}\quad Y_m(t)
  &=\frac2\pi\left(\log\frac t2+\gamma_{\rm E}\right)J_m(t)
    -\frac1\pi
    \sum_{j=0}^{m-1}\frac{(m-j-1)!}{j!}
    \left(\frac{t}{2}\right)^{2j-m} \notag\\
  &\quad
    -\frac1\pi\sum_{j=0}^{+\infty}
    \frac{(-1)^j(H_j+H_{m+j})}{j!(m+j)!}
    \left(\frac t2\right)^{m+2j},
  \label{eq:Yn-ascending}
\end{align}
where $\gamma_{\rm E}$ is Euler's constant, $H_0:=0$, and $H_m:=\sum_{\ell=1}^m\ell^{-1}$ for $m\ge1$.
\end{lemma}
\begin{proof}
Setting $\nu = m$ in the defining series \eqref{eq:J-definition} and using $\Gamma(j + m + 1) = (m + j)!$, we obtain \eqref{eq:Jn-ascending}.
The series expansion \eqref{eq:Yn-ascending} for $Y_m$ follows from \cite[Section~3.52, Eq.~(3)]{watson1944} under the normalization in \cite[Section~3.54, Eq.~(2)]{watson1944}.
\end{proof} 
 
We use the following fixed-order expansions to obtain the objective-gap asymptotic in Corollary~\ref{coro:continuous-counterexample}.

\begin{lemma}[\cite{watson1944}~Section~7.21, Eqs.~(1) and~(5)]
\label{lem:bessel-asymptotics}
For a fixed integer $m\ge1$, set $\phi_m(t):=t-\frac{m\pi}{2}-\frac{\pi}{4}$.
Then, as $t\to+\infty$, we have
\begin{align}
  J_m(t)
  &=\sqrt{\frac{2}{\pi t}}
    \Big(
      \cos\phi_m(t)-\frac{4m^2-1}{8t}\sin\phi_m(t)
      +O(t^{-2})
    \Big)
  \label{eq:Jn-large}
  \\
  \mbox{and}\quad Y_m(t)
  &=\sqrt{\frac{2}{\pi t}}
    \Big(
      \sin\phi_m(t)+\frac{4m^2-1}{8t}\cos\phi_m(t)
      +O(t^{-2})
    \Big).
  \label{eq:Yn-large}
\end{align}
\end{lemma}

Finally, we recall Nicholson's product formula for the sum of squares $J_m^2 + Y_m^2$.
\begin{lemma}[Nicholson's formula \cite{watson1944}]
\label{lem:bessel-nicholson}
For every integer $m\ge1$ and $t>0$, we have
\begin{equation}
  J_m(t)^2+Y_m(t)^2
  =\frac8{\pi^2}\int_0^{+\infty}
    K_0(2t\sinh r)\cosh(2mr)\dd r
  \label{eq:bessel-nicholson}
\end{equation}
where the function $K_0:(0,+\infty)\to\R$ is defined by
\begin{equation}
  K_0(v):=\int_0^{+\infty}\ee^{-v\cosh s}\dd s,
  \qquad \text{for }v>0.
  \label{eq:K0-integral}
\end{equation}
\end{lemma}

\begin{proof}
The identity \eqref{eq:bessel-nicholson} follows from \cite[Section~13.73, Eq.~(1)]{watson1944} with order $m$, while the integral representation \eqref{eq:K0-integral} follows from Watson's formula in \cite[Section~6.22, Eq.~(5)]{watson1944} with order $0$.
\end{proof}

\section{Proof of Theorem \ref{thm:discrete-counterexample}}
\label{sec:discrete}
In this section, we establish Theorem~\ref{thm:discrete-counterexample} by constructing and analyzing a two-dimensional discrete counterexample.
We prescribe the sequence $\{Z_k\}$ in \eqref{eq:discrete-tail} using the associated Ferrers functions of order $-1$ defined in \eqref{eq:negative-order-identities}.
Using Lemma~\ref{lem:ferrers-identities}, we verify that this sequence satisfies the quadratic NAG recurrence \eqref{eq:tail-recurrence}.
We then use the integral representation in Lemma~\ref{lem:ferrers-nicholson} to prove that $\langle Z_k,Z_k-Z_{k-1}\rangle<0$ for every $k\ge1$.
However, because this target sequence involves a second-kind Ferrers component, it is incompatible with the standard zero-momentum initialization under a purely quadratic objective.
To resolve this obstruction, we modify the second component of the quadratic objective outside a half-space containing the origin to match the initial steps.

We construct the objective function and the initial point in Subsection~\ref{subsec:discrete-construction}.
In Subsection~\ref{subsec:discrete-verification}, we verify that the objective belongs to $\mathcal{S}_{\mu,L}(\mathbb{R}^2)$ and that the resulting NAG iterates never trigger gradient restart.

\subsection{Construction of the discrete counterexample}
\label{subsec:discrete-construction}

We construct the objective $f_{\mathrm{d}}$ by adding a univariate perturbation to an isotropic quadratic $(\xi, v) \mapsto \mu_{\mathrm{d}}(\xi^2 + v^2)/2$ with $\mu_{\mathrm{d}} > 0$ along its second coordinate.
Specifically, for two thresholds $c < U$, we design the perturbation $-h_{\mathrm{d}}(v)$ of the second gradient component to vanish on the half-space $v \le c$, transition smoothly on $[c, U]$, and remain constant for $v \ge U$.
We first fix the step size $s:=1$ and the constants
\begin{equation}
  a:=\frac9{10},
  \qquad
  \mu_{\rm d}:=\frac1{10},
  \qquad
  \rho:=\frac3{\sqrt{10}},
  \qquad
  c:=\frac18,
  \qquad\mbox{and}\qquad
  \varepsilon:=\frac1{1000}.
  \label{eq:discrete-parameters}
\end{equation}
Here $\mu_{\rm d}$ is the quadratic curvature, $a=1-\mu_{\rm d}$ is the corresponding gradient-step factor, and $\rho=\sqrt a\in(0,1)$ is the Ferrers parameter.
The constants $c$ and $\varepsilon$ specify the quadratic region and the transition widths, respectively.
Set
\begin{equation}
  U:=\frac{\rho+\mu_{\rm d}\operatorname{arctanh}\rho}{\pi}
  \qquad\mbox{and}\qquad
  d:=\frac{2\rho}{3\pi}=\frac{2}{\pi\sqrt{10}}.
  \label{eq:U-d}
\end{equation}
The constants $U$ and $d$ are chosen to match the first two terms of the target sequence.
To construct $h_{\rm d}$, we first introduce the continuous piecewise-linear function $\sigma:\R\to[0,1]$ supported on $[c,U]$ by
\begin{equation}
  \sigma(t):=
  \begin{cases}
    0,&t\le c,\\
    (t-c)/\varepsilon,&c<t<c+\varepsilon,\\
    1,&c+\varepsilon\le t\le U-\varepsilon,\\
    (U-t)/\varepsilon,&U-\varepsilon<t<U,\\
    0,&t\ge U.
  \end{cases}
  \label{eq:sigma-definition}
\end{equation}
Then $h_{\rm d}:\R\to\R$ is defined by
\begin{equation}
  h_{\rm d}(t):=-b_{\rm d}\int_{-\infty}^t\sigma(r)\dd r
  \qquad\mbox{and}\qquad
  b_{\rm d}:=\frac{d}{U-c-\varepsilon}.
  \label{eq:discrete-h}
\end{equation}
Since $\sigma \ge 0$ is supported on $[c, U]$ and satisfies $\int_c^U \sigma(r)\,\mathrm{d}r = U - c - \varepsilon$ by \eqref{eq:sigma-definition}, it follows from \eqref{eq:discrete-h} that $h_{\mathrm{d}}$ decreases from $0$ to $-d$ on $[c, U]$ and remains constant outside this interval.
Finally, we define the objective function $f_{\rm d}:\R^2\to\R$ by
\begin{equation}
  f_{\rm d}(\xi,v)
  :=
  \frac{\mu_{\rm d}}{2}(\xi^2+v^2)
  -\int_0^v h_{\rm d}(t)\dd t,
  \label{eq:discrete-objective}
\end{equation}
and the initial point by
\begin{equation}
  x_0:=y_0
  :=
  \left(
    \frac{\mu_{\rm d}}{2a},
    \frac{U+d}{a}
  \right).
  \label{eq:discrete-initial-point}
\end{equation}
For $s = 1$, the gradient-step map is given by
\begin{equation}
\label{eq:discrete-gradient-map}
  \mathcal{T}(\xi, v) := (\xi, v) - \nabla f_{\mathrm{d}}(\xi, v) = (a\xi, av + h_{\mathrm{d}}(v)).
\end{equation}
Since $h_{\mathrm{d}}(v) = 0$ for $v \le c$, the objective $f_{\mathrm{d}}$ coincides with $\mu_{\mathrm{d}}(\xi^2 + v^2)/2$ on this half-space, where $\mathcal{T}$ reduces to scalar multiplication by $a$ and the NAG updates follow the unmodified quadratic recurrence.
Above the threshold $U$, the constant shift $h_{\mathrm{d}}(v) = -d$ and the initial point $x_0$ are calibrated so that the zero-momentum initialization produces the first two terms of the target sequence defined in \eqref{eq:discrete-tail} and drives the subsequent iterates into the quadratic region $v \le c$.
Before proceeding to the formal verification, we present a numerical evaluation in Figure~\ref{fig:discrete-counterexample} to illustrate the strict negativity of the restart quantities $D_k$ up to $k = 10^4$.

\begin{figure}[!ht]
  \centering
  \includegraphics[width=0.8\textwidth]{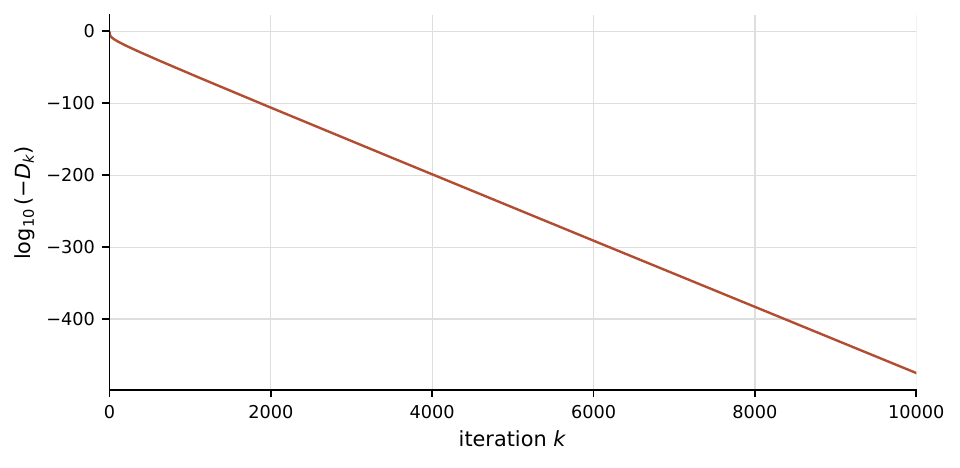}
\caption{Numerical evaluation of the discrete construction for $0\le k\le10^4$, computed with 100-digit arithmetic.
The curve shows $\log_{10}(-D_k)$, which visually illustrates the strict negativity proved in Theorem~\ref{thm:discrete-counterexample}.}
  \label{fig:discrete-counterexample}
\end{figure}

\subsection{Verification of the discrete counterexample}
\label{subsec:discrete-verification}
Whenever the extrapolated iterates lie in the half-space $v \le c$, the NAG updates reduce to the isotropic quadratic recurrence.
Motivated by the Ferrers degree recurrence in Lemma~\ref{lem:ferrers-identities}, we define $\Phi_m : (-1, 1) \to \R^2$ for each integer $m \ge 1$ by
\begin{equation}
\label{eq:Phi-definition}
  \Phi_m(x) := \left(P_m^{(-1)}(x), \frac{2}{\pi}Q_m^{(-1)}(x)\right),
\end{equation}
where $P_m^{(-1)}$ and $Q_m^{(-1)}$ are the associated Ferrers functions of order $-1$ defined in~\eqref{eq:negative-order-identities}, and set
\begin{equation}
\label{eq:discrete-tail}
  Z_k := \rho^k \sqrt{\mu_{\mathrm{d}}}\, \Phi_{k+1}(\rho) \quad \text{for } k \ge 0.
\end{equation}
In the remainder of this subsection, we first establish the recurrence and radial monotonicity of $\{Z_k\}$ in Lemma~\ref{lem:ferrers-tail}, and then prove in Lemma~\ref{lem:discrete-orbit} that the NAG iterates satisfy $x_{k+1} = Z_k$ for $k \ge 0$ and $y_k = Z_k/a$ with $y_k$ remaining in the half-space $v \le c$ for $k \ge 2$.
We then use these results to verify that $f_{\mathrm{d}} \in \mathcal{S}_{\mu,L}(\mathbb{R}^2)$ and that $D_k < 0$ for all $k \ge 0$, thereby proving Theorem~\ref{thm:discrete-counterexample}.

We first establish several elementary parameter bounds to verify that $h_{\mathrm{d}}$ is well defined with $0 < b_{\mathrm{d}} < a$ and that the extrapolated point $y_2$ lies in the half-space $v \le c$.

\begin{lemma}
\label{lem:discrete-parameters}
The constants in \eqref{eq:discrete-parameters}--\eqref{eq:U-d} and the sequence $\{Z_k\}$ in \eqref{eq:discrete-tail} satisfy
\begin{equation}
  U>c+2\varepsilon,
  \qquad
  0<d<a(U-c-\varepsilon),
  \qquad\mbox{and}\qquad
  \norm{5Z_1-Z_0}<4c.
  \label{eq:discrete-parameter-bounds}
\end{equation}
\end{lemma}

The proof is given in Appendix~\ref{app:discrete-parameter-check}.
We next establish the recurrence and monotonicity properties of the sequence $\{Z_k\}$.

\begin{lemma}
\label{lem:ferrers-tail}
For the constants $a,\mu_{\rm d},\rho$ in \eqref{eq:discrete-parameters} and $U,d$ in \eqref{eq:U-d}, the sequence $\{Z_k\}$ defined by \eqref{eq:discrete-tail} has the following properties.
\begin{enumerate}
\item[(i)] The sequence $\{Z_k\}$ satisfies the recurrence
\begin{equation}
  Z_{k+1}
  =\frac{a}{k+3}\bigl((2k+3)Z_k-kZ_{k-1}\bigr)
  \qquad \text{for }k\ge1,
  \label{eq:tail-recurrence}
\end{equation}
with initial terms
\begin{equation}
  Z_0=\left(\frac{\mu_{\rm d}}{2},U\right)
  \qquad\mbox{and}\qquad
  Z_1=\left(\frac{a\mu_{\rm d}}{2},aU-d\right).
  \label{eq:tail-first-two}
\end{equation}

\item[(ii)] For every $k\ge1$, the sequence $\{Z_k\}$ satisfies
\begin{equation}
  \ip{Z_k}{Z_k-Z_{k-1}}<0.
  \label{eq:tail-inner-product-negative}
\end{equation}
Moreover, we have
\begin{equation}
  \norm{Z_k}<\norm{Z_{k-1}},
  \qquad \text{for every }k\ge1.
  \label{eq:tail-radial-decrease}
\end{equation}

\end{enumerate}
\end{lemma}

\begin{proof}
For (i), each coordinate of $\Phi_m$ satisfies the degree recurrence \eqref{eq:ferrers-degree-recurrence} in Lemma~\ref{lem:ferrers-identities}.
With $m=k+1$ and $x=\rho$, this recurrence reads, for $k\ge1$,
\begin{equation*}
    (k+3)\Phi_{k+2}(\rho)
  =(2k+3)\rho\Phi_{k+1}(\rho)-k\Phi_k(\rho).
\end{equation*}
Multiplying both sides of this recurrence by $\rho^{k+1}\sqrt{\mu_{\rm d}}$ and using \eqref{eq:discrete-tail}, we obtain
\begin{align*}
  (k+3)Z_{k+1}
  &=\rho^{k+1}\sqrt{\mu_{\rm d}}
    \bigl((2k+3)\rho\Phi_{k+1}(\rho)-k\Phi_k(\rho)\bigr)\\
  &=\rho^2\bigl((2k+3)Z_k-kZ_{k-1}\bigr).
\end{align*}
Since $\rho^2=a$, this proves \eqref{eq:tail-recurrence}.
Differentiating Rodrigues' formula \eqref{eq:legendre-rodrigues} for $P_1,P_2$ and \eqref{eq:second-kind-Q-definition} for $Q_1,Q_2$, we obtain, for $-1<x<1$,
\begin{align*}
  P_1'(x)&=1,
  &Q_1'(x)&=\operatorname{arctanh}x+\frac{x}{1-x^2},
  \\
  P_2'(x)&=3x,
  &\mbox{and}\quad Q_2'(x)&=3xQ_1'(x)-\frac{2}{1-x^2}. 
\end{align*}
Using these derivative formulas and \eqref{eq:negative-order-identities}, with $1-\rho^2=\mu_{\rm d}$, we evaluate the first two terms of \eqref{eq:discrete-tail}.
For $m=1$, we obtain
\begin{align*}
  Z_0&=\sqrt{\mu_{\rm d}}
       \left(P_1^{(-1)}(\rho),\frac2\pi Q_1^{(-1)}(\rho)\right)
       =\left(\frac{\mu_{\rm d}}{2}P_1'(\rho),
             \frac{\mu_{\rm d}}{\pi}Q_1'(\rho)\right)=\left(\frac{\mu_{\rm d}}{2},
             \frac{\rho+\mu_{\rm d}\operatorname{arctanh}\rho}{\pi}\right)
      =\left(\frac{\mu_{\rm d}}{2},U\right).
\end{align*}
Here the last equality uses \eqref{eq:U-d}.
For $m=2$, using $\rho^2=a$, we similarly obtain
\begin{align*}
  Z_1&=\rho\sqrt{\mu_{\rm d}}
       \left(P_2^{(-1)}(\rho),\frac2\pi Q_2^{(-1)}(\rho)\right)=\left(\frac{\rho\mu_{\rm d}}{6}P_2'(\rho),
             \frac{\rho\mu_{\rm d}}{3\pi}Q_2'(\rho)\right)\\
     &=\left(\frac{a\mu_{\rm d}}{2},
             \frac{a\mu_{\rm d}}{\pi}Q_1'(\rho)-\frac{2\rho}{3\pi}\right)
     =\left(\frac{a\mu_{\rm d}}{2},aU-d\right).
\end{align*}
The last equality uses $\mu_{\rm d}Q_1'(\rho)/\pi=U$ and $2\rho/(3\pi)=d$.
This proves \eqref{eq:tail-first-two} and completes (i).

For (ii), fix $k\ge1$.
Substituting the formulas for $Z_k$ and $Z_{k-1}$ from \eqref{eq:discrete-tail} into their difference and applying the differentiation identity \eqref{eq:ferrers-differentiation} in Lemma~\ref{lem:ferrers-identities} with $m=k+1$ and $x=\rho$, we obtain
\[
  Z_k-Z_{k-1}
  =-\frac{\rho^{k-1}\sqrt{\mu_{\rm d}}}{k}
    \left(\rho\Phi_{k+1}(\rho)+\mu_{\rm d}\Phi_{k+1}'(\rho)\right).
\]
For integers $m\ge2$, define $M_m(x):=P_m'(x)^2+(4/\pi^2)Q_m'(x)^2$ on $(0,1)$.
Taking the squared norm in \eqref{eq:Phi-definition} and substituting the order-$-1$ identities \eqref{eq:negative-order-identities}, we obtain
\[
  \norm{\Phi_m(x)}^2
  =\frac{1-x^2}{m^2(m+1)^2}M_m(x).
\]
Differentiating the expression for $\norm{\Phi_m(x)}^2$ and rearranging, we have
\[
  x\norm{\Phi_m(x)}^2
  +(1-x^2)\ip{\Phi_m(x)}{\Phi_m'(x)}
  =\frac{(1-x^2)^2}{2m^2(m+1)^2}M_m'(x).
\]
Taking the inner product of $Z_k=\rho^k\sqrt{\mu_{\rm d}}\Phi_{k+1}(\rho)$ with the expression for $Z_k-Z_{k-1}$, and substituting the derivative identity for $\norm{\Phi_m}^2$ with $m=k+1$ and $x=\rho$, we obtain
\begin{equation}
  \ip{Z_k}{Z_k-Z_{k-1}}
  =
  -\frac{\rho^{2k-1}(1-\rho^2)^3}
         {2k(k+1)^2(k+2)^2}
  M_{k+1}'(\rho).
  \label{eq:tail-inner-product-identity}
\end{equation}

It remains to prove that $M_m'(x)>0$ for every $m\ge2$ and $0<x<1$.
We establish this inequality using the Nicholson-type integral representation.
Define $B_m(w) := -\sqrt{w^2 - 1}\,\widehat{Q}_m'(w)$ for $w > 1$.
By Lemma~\ref{lem:ferrers-nicholson}, we have $B_m(w) > 0$ for all $w > 1$.
From \eqref{eq:ferrers-derivative-nicholson} in Lemma~\ref{lem:ferrers-nicholson}, we obtain
\[
  M_m(x)=\frac{4m(m+1)}{\pi^2}
    \int_1^{+\infty} B_m(w)K(x,w)\dd w
  \quad \mbox{with} \quad
  K(x,w):=(1-x^2)^{-2}\sqrt{1-\frac{2x^2}{w+1}}.
\]
The dependence of the integrand on $x$ is entirely through $K$, whose logarithmic derivative satisfies, for $0<x<1$ and $w>1$,
\[
  \partial_x\log K(x,w)
  =\frac{2x(2w+1-3x^2)}
         {(1-x^2)(w+1-2x^2)}>0.
\] 
Since $K(x,w) \ge (1 - x^2)^{-3/2}$, we deduce from the absolute convergence in Lemma~\ref{lem:ferrers-nicholson} that $B_m \in L^1(1, +\infty)$.
For every compact $E\subset(0,1)$, the functions $K$ and $\partial_xK$ are uniformly bounded on $E\times[1,+\infty)$.
Therefore, by the dominated convergence theorem, we obtain
\[
  M_m'(x)=\frac{4m(m+1)}{\pi^2}
    \int_1^{+\infty} B_m(w)\partial_xK(x,w)\dd w.
\]
Since $B_m>0$ and $\partial_xK(x,w)>0$, it follows that $M_m'(x)>0$.
Since $M_{k+1}'(\rho)>0$ and the coefficient in \eqref{eq:tail-inner-product-identity} is negative, we obtain \eqref{eq:tail-inner-product-negative}.
Moreover,
\begin{equation*}
  \norm{Z_k}^2-\norm{Z_{k-1}}^2
  =2\ip{Z_k}{Z_k-Z_{k-1}}
   -\norm{Z_k-Z_{k-1}}^2<0,
\end{equation*}
which proves \eqref{eq:tail-radial-decrease} and completes (ii).
\end{proof}

We now show that the prescribed initial point generates the sequence $\{x_{k+1}\}$ which coincides with $\{Z_k\}$.

\begin{lemma}
\label{lem:discrete-orbit}
Let $\{(x_k,y_k)\}$ be generated by the NAG iteration \eqref{eq:discrete-nag} with $f=f_{\rm d}$, $s=1$, and the initialization \eqref{eq:discrete-initial-point}.
Then the iterates $\{x_k\}$ and $\{y_k\}$ satisfy
\begin{equation*}
  x_1=y_1=Z_0
  \qquad\mbox{and}\qquad
  x_2=Z_1,
\end{equation*}
and, for every $k\ge2$,
\begin{equation}
  x_{k+1}=Z_k,
  \qquad
  y_k=\frac{1}{a}Z_k,
  \qquad\mbox{and}\qquad
  \norm{y_k}\le\norm{y_2}<c.
  \label{eq:discrete-orbit-identification}
\end{equation}
\end{lemma}

\begin{proof}
Since $\sigma(t) = 0$ for $t \notin (c, U)$ and $\int_{\mathbb{R}} \sigma(t)\,\mathrm{d}t = U - c - \varepsilon$, we have from the definition of $b_{\mathrm{d}}$ in \eqref{eq:discrete-h} that
\begin{equation}
  h_{\rm d}(t)=0 \quad\text{for }t\le c
  \qquad\mbox{and}\qquad
  h_{\rm d}(t)=-d \quad\text{for }t\ge U.
  \label{eq:hd-properties}
\end{equation}
Since $c>0$ and $h_{\rm d}$ vanishes on $(-\infty,c]$, the integral in \eqref{eq:discrete-objective} vanishes whenever $v\le c$.
Hence
\begin{equation}
  f_{\rm d}(\xi,v)=\frac{\mu_{\rm d}}{2}(\xi^2+v^2)
  \qquad \text{for }(\xi,v)\in\R^2\text{ with }v\le c.
  \label{eq:discrete-quadratic-half-space}
\end{equation}

The initial point in \eqref{eq:discrete-initial-point} satisfies $(x_0)_2=(U+d)/a>U$, so $h_{\rm d}((x_0)_2)=-d$ by \eqref{eq:hd-properties}.
Substituting $x_0$ into the gradient-step map \eqref{eq:discrete-gradient-map}, we obtain $x_1=Z_0$, with $Z_0$ as evaluated in Lemma~\ref{lem:ferrers-tail}(i).
The first momentum coefficient in \eqref{eq:discrete-nag} is zero, so $y_1=x_1=Z_0$.
Applying the gradient-step map \eqref{eq:discrete-gradient-map} to $y_1$, whose second coordinate equals $U$, we obtain $x_2 = Z_1$ since $h_{\mathrm{d}}(U) = -d$.
Hence
\begin{align*}
  x_1=y_1=Z_0,\qquad x_2=aZ_0-(0,d)=Z_1,\quad
  \mbox{and}\quad y_2=x_2+\frac14(x_2-x_1)=\frac{5Z_1-Z_0}{4}.
\end{align*}
By the recurrence~\eqref{eq:tail-recurrence} with $k = 1$, we have $y_2 = (5Z_1 - Z_0)/4 = Z_2/a$.
With the estimate~\eqref{eq:discrete-parameter-bounds} in Lemma~\ref{lem:discrete-parameters}, we obtain $\norm{y_2}<c$, so $h_{\mathrm{d}}((y_2)_2) = 0$ and $x_3 = ay_2 = Z_2$.
This establishes~\eqref{eq:discrete-orbit-identification} for $k = 2$.

To continue inductively, suppose that $k\ge3$, $x_k=Z_{k-1}$, and $x_{k-1}=Z_{k-2}$.
Combining the momentum update in \eqref{eq:discrete-nag} with the recurrence \eqref{eq:tail-recurrence} from Lemma~\ref{lem:ferrers-tail}(i), we obtain
\begin{align*}
  y_k
  =x_k+\frac{k-1}{k+2}(x_k-x_{k-1})
  =\frac{(2k+1)Z_{k-1}-(k-1)Z_{k-2}}{k+2}
    =\frac{Z_k}{a}.
\end{align*}
Since $\norm{Z_k}$ decreases by Lemma~\ref{lem:ferrers-tail}(ii), the extrapolated point $y_k=Z_k/a$ satisfies
\begin{equation*}
  \norm{y_k}
  =\frac{1}{a}\norm{Z_k}
  \le\frac{1}{a}\norm{Z_2}
  =\norm{y_2}
  <c.
\end{equation*}
Thus $(y_k)_2<c$, so $h_{\rm d}((y_k)_2)=0$, and therefore $x_{k+1}=ay_k=Z_k$.
This completes the induction for \eqref{eq:discrete-orbit-identification}.

\end{proof}

To complete the proof of Theorem~\ref{thm:discrete-counterexample}, it remains to prove $f_{\rm d}\in\mathcal S_{\mu,L}(\R^2)$ for constants $0<\mu\le L<1=1/s$, and verify the non-activation property \eqref{eq:discrete-main-nonrestart}.

\begin{proof}[Proof of Theorem~\ref{thm:discrete-counterexample}]
Take $f:=f_{\rm d}$, $s:=1$, and the initial point \eqref{eq:discrete-initial-point}.
Since $\sigma$ is continuous, $h_{\rm d}\in C^1(\R)$ with $h_{\rm d}'=-b_{\rm d}\sigma$, and hence $f\in C^2(\R^2)$.
Differentiating the objective $f_{\rm d}$ in \eqref{eq:discrete-objective}, we obtain the Hessian
\[
  \nabla^2 f(\xi,v)
  =\operatorname{diag}\bigl(\mu_{\rm d},\mu_{\rm d}+b_{\rm d}\sigma(v)\bigr).
\]
By Lemma~\ref{lem:discrete-parameters}, we have $0<d<a(U-c-\varepsilon)$, so the constant $b_{\rm d}$ in \eqref{eq:discrete-h} satisfies $0<b_{\rm d}<a$.
Set $L_{\rm d}:=\mu_{\rm d}+b_{\rm d}<\mu_{\rm d}+a=1=1/s$.
Since $0\le\sigma\le1$, both diagonal entries of the Hessian lie between $\mu_{\rm d}$ and $L_{\rm d}$, that is,
\begin{equation}
  \mu_{\rm d} I\preceq\nabla^2 f(\xi,v)\preceq L_{\rm d}I
  \qquad \text{for all } (\xi,v)\in\R^2.
  \label{eq:discrete-Hessian-bound}
\end{equation}
Thus $f\in\mathcal S_{\mu,L}(\R^2)$ and $s=1<1/L$, with $\mu:=\mu_{\rm d}$ and $L:=L_{\rm d}$.
Since $h_{\rm d}(0)=0$, we have $\nabla f(0)=0$, and by strong convexity we conclude that $x^\star=0$ is the unique minimizer.
The first coordinate of $x_0$ is $\mu_{\rm d}/(2a)>0$, so the initialization is nonoptimal.

Finally, it suffices to verify \eqref{eq:discrete-main-nonrestart}, i.e., $D_k<0$ for all $k\ge0$.
Since $y_0=x_0$ and $y_1=x_1$, we obtain from the definition \eqref{eq:discrete-gr} that
\begin{equation*}
  D_0=-\norm{\nabla f_{\rm d}(y_0)}^2<0
  \qquad\mbox{and}\qquad
  D_1=-\norm{\nabla f_{\rm d}(y_1)}^2<0,
\end{equation*}
where the inequalities are strict because $(\nabla f_{\rm d}(y_0))_1=\mu_{\rm d}^2/(2a)>0$ and $(\nabla f_{\rm d}(y_1))_1=\mu_{\rm d}^2/2>0$, using the initial point \eqref{eq:discrete-initial-point} and the formula for $Z_0$ in Lemma~\ref{lem:ferrers-tail}(i).

For $k \ge 2$, it follows from Lemma~\ref{lem:discrete-orbit}, specifically \eqref{eq:discrete-orbit-identification}, that $x_{k+1}=Z_k$ and $y_k=Z_k/a$, with $y_k$ in the half-space $v\le c$.
On this half-space, $f_{\rm d}$ has the quadratic form \eqref{eq:discrete-quadratic-half-space}, and therefore
\begin{equation*}
  \nabla f_{\rm d}(y_k)=\mu_{\rm d} y_k=\frac{\mu_{\rm d}}{a}Z_k
  \qquad\mbox{and}\qquad
  x_{k+1}-x_k=Z_k-Z_{k-1}.
\end{equation*}
The restart quantity in \eqref{eq:discrete-gr} is thus a positive multiple of the inner product in Lemma~\ref{lem:ferrers-tail}(ii), namely,
\begin{equation*}
  D_k
  =\frac{\mu_{\rm d}}{a}\ip{Z_k}{Z_k-Z_{k-1}}
  <0,
\end{equation*}
which completes the proof.

\end{proof}

\begin{remark}
From the asymptotic expansions \eqref{eq:ferrers-P-large} and \eqref{eq:ferrers-Q-large} in Lemma~\ref{lem:ferrers-asymptotics}, we have $\|\Phi_{k+1}(\rho)\| = O(k^{-3/2})$, and thus $\|Z_k\| = O(\rho^k k^{-3/2})$ by \eqref{eq:discrete-tail}.
Applying the orbit identification \eqref{eq:discrete-orbit-identification} in Lemma~\ref{lem:discrete-orbit}, we obtain that both $\|x_{k+1}\|^2$ and $\|y_k\|^2$ decay at the asymptotic rate $O(\rho^{2k}k^{-3})$.
Since $0<\rho<1$ and no restart occurs, both the restarted and unrestarted iterates converge $R$-linearly to zero.
\end{remark}

\begin{remark}
\label{rem:discrete-extensions}
The discrete counterexample admits the following two extensions.

\begin{enumerate}
\item[(i)] The construction can be modified to obtain a counterexample with step size $s=1/L$.
Define $F:\R^3\to\R$ by
\begin{equation*}
  F(\xi,v,w)
  :=f_{\rm d}(\xi,v)+\frac12w^2.
\end{equation*}
From the Hessian bound \eqref{eq:discrete-Hessian-bound} in the proof of Theorem~\ref{thm:discrete-counterexample}, we have
\begin{equation*}
  \mu_{\rm d} I_3\preceq\nabla^2F\preceq I_3,
\end{equation*}
and the exact global Lipschitz constant of $\nabla F$ is $L=1$.
For the initialization $\widetilde x_0:=\widetilde y_0:=(x_0,0)$ and the step size $s=1$, the third coordinate remains zero, so the first two coordinates reproduce the two-dimensional counterexample.
Hence $D_k<0$ for every $k\ge0$ also when $s=1/L$.

\item[(ii)] The objective $f_{\rm d}$ in \eqref{eq:discrete-objective} can be chosen in $C^\infty$ without changing its NAG iterates.
Replace the piecewise-linear function $\sigma$ in \eqref{eq:sigma-definition} by a function in $C_c^{\infty}((c,U))$ satisfying $0\le\sigma\le1$ and $\int_c^U\sigma(t)\dd t>d/a$, and set $b_{\rm d}:=d/\int_c^U\sigma(t)\dd t$.
By Lemma~\ref{lem:discrete-parameters}, we have $U - c > d/a$, so such a smooth function $\sigma$ exists and the resulting constant $b_{\mathrm{d}}$ satisfies $b_{\mathrm{d}} < a$.
Using the smooth function $\sigma$ and the constant $b_{\rm d}$ chosen above in \eqref{eq:discrete-h} and \eqref{eq:discrete-objective}, we obtain infinitely differentiable functions $h_{\rm d}$ and $f_{\rm d}$.
Because $h_{\rm d}$ still vanishes on $(-\infty,c]$ and equals $-d$ on $[U,+\infty)$, the proof of Lemma~\ref{lem:discrete-orbit} applies with the same gradient updates, so the iterates $\{(x_k,y_k)\}$ and restart quantities $D_k$ are unchanged.
\end{enumerate}
\end{remark}

\section{Proof of Theorem \ref{thm:continuous-counterexample} and Corollary~\ref{coro:continuous-counterexample}}
\label{sec:continuous}

In this section, we establish Theorem~\ref{thm:continuous-counterexample} and Corollary~\ref{coro:continuous-counterexample} by constructing the continuous-time counterexample in Subsection~\ref{subsec:continuous-construction} and verifying its properties in Subsection~\ref{subsec:continuous-prove}.
In parallel with the discrete construction, we begin with a two-dimensional solution of the isotropic quadratic NAG-ODE, formed by scaled Bessel functions of the first and second kinds.
Because the second-kind component diverges at $t = 0$, we replace it on a short initial interval $[0, T]$ by a smooth curve with zero initial velocity.
To make this modified trajectory solve the NAG-ODE for a single objective in $\mathcal{S}_{\mu, L}(\R^2)$, we prescribe the second gradient component along the initial segment and extend it smoothly to $\R$ with a positive, bounded derivative.

\subsection{Construction of the continuous counterexample}
\label{subsec:continuous-construction}
For the isotropic quadratic objective $(x, y) \mapsto (x^2 + y^2)/2$, each coordinate of the NAG-ODE~\eqref{eq:nag-ode} satisfies the scalar equation
\begin{equation}
\label{eq:scalar-bessel-equation}
  q''(t) + \frac{3}{t}q'(t) + q(t) = 0, \qquad \text{for } t > 0.
\end{equation}
As noted in Subsection~\ref{subsec:bessel-preliminaries}, the substitution $w(t) = tq(t)$ reduces~\eqref{eq:scalar-bessel-equation} to Bessel's equation~\eqref{eq:bessel-equation}, so the scaled Bessel functions
\begin{equation}
\label{eq:bessel-branches}
  u(t) := \frac{J_1(t)}{t}, \qquad \eta(t) := -\frac{Y_1(t)}{t}, \qquad \text{for } t > 0,
\end{equation}
form a fundamental system of solutions.
We pair $u$ with $\eta$ because, by Nicholson's formula \eqref{eq:bessel-nicholson} in Lemma~\ref{lem:bessel-nicholson}, we have $u(t)u'(t) + \eta(t)\eta'(t) < 0$ for all $t > 0$.
By the series expansions~\eqref{eq:Jn-ascending}--\eqref{eq:Yn-ascending} in Lemma~\ref{lem:bessel-ascending}, $u$ extends smoothly to $t = 0$ with $u(0) = 1/2$ and $u'(0) = 0$, whereas $\eta(t)$ diverges as $t \downarrow 0$.

To obtain a valid trajectory starting from rest, we fix $T:= 1/10$, retain $u$ on $[0, \infty)$, and replace $\eta$ on $[0, T]$ by a smooth connecting function $p$ that satisfies $p'(0) = 0$ and agrees with $\eta$ to all orders at $T$.
Let $h(t)$ denote the prescribed second gradient component along $(u(t), p(t))$ on $[0, T]$, which is coupled with $p$ through the ODE $\frac{\dd }{\dd t} (t^3p')+t^3 h=0$ on $(0, T]$.
To decouple this relation, we introduce an auxiliary function $g:= -h'$, from which both $h$ and $p$ can be recovered by successive integration.
To meet these requirements, we enforce zero initial velocity, smooth matching at $T$, and the required restart sign through endpoint, positivity, and weighted moment conditions on $g$.
We fulfill these conditions by smoothly interpolating between $2t$ and $-\eta'(t)$ and adding a localized moment correction.

For this interpolation, define the smooth transition function $\theta_T:\R\to[0,1]$ by
\begin{equation}
  \theta_T(t):=\vartheta(2t/T-1)\quad \text{with}\quad\vartheta(r):=
  \frac{\chi(r)}{\chi(r)+\chi(1-r)}, \quad \text{and}\quad\chi(r):=
  \begin{cases}
    0,&r\le0,\\
    \exp(-1/r),&r>0.
  \end{cases}
  \label{eq:smooth-step}
\end{equation}
Using this transition, $g_{\rm b}:[0,T]\to\R$ is defined to smoothly connect $2t$ and $-\eta'(t)$, namely,
\begin{equation}
  g_{\rm b}(t):=
  \begin{cases}
    2t,&0\le t\le T/2,
    \\[1mm]
    (1-\theta_T(t))\,2t-\theta_T(t)\eta'(t),
      &T/2<t\le T.
  \end{cases}
  \label{eq:continuous-baseline}
\end{equation}
Since the terminal condition $p'(T) = \eta'(T)$ requires a prescribed weighted moment of $g$, we define the constants $I_T$, $R_T$, and $N_T$ and the function $\psi_T$ by
\begin{align}
  &I_T:=\int_0^T s^4\theta_T'(s)\dd s,
  &&R_T:=-4T^3\eta'(T)-T^4\eta(T),
  \label{eq:IT-RT}
  \\
   &N_T:=R_T-\int_0^T s^4g_{\rm b}(s)\dd s &\mbox{and}\quad &\psi_T(t):=\frac{\theta_T'(t)}{I_T}.
  \label{eq:NT-definition}
\end{align}
Since $\theta_T'=0$ on $(0,T/2]$ and $\theta_T'>0$ on $(T/2,T)$, we have $I_T>0$ and $\int_0^T s^4\psi_T(s)\dd s=1$.

Define $g:[0,T]\to\R$ as
\begin{equation}
\label{eq:g-definition}
  g(t):=g_{\rm b}(t)+N_T\psi_T(t)
\end{equation}
Then $g$ will satisfy the initial condition, endpoint matching condition, and the moment condition $\int_0^T s^4g(s)\dd s=R_T$ needed for constructing the connecting function $p$ and the prescribed second gradient component $h$.
After that, $h,p:[0,T]\to\R$ are defined by
\begin{align}
  h(t)&:=\eta(T)+\int_t^Tg(s)\dd s
  \label{eq:h-definition}
  \\
  \mbox{and}\quad p(t)&:=\eta(T)+\int_t^T r^{-3}
  \left(\int_0^r s^3h(s)\dd s\right)\dd r,
  \label{eq:p-definition}
\end{align}
such that they satisfy $h'=-g$ and $p''+3p'/t+h=0$ on $(0,T]$, with $h(T)=p(T)=\eta(T)$ and $p'(0)=0$.
Here we write $h_0 := h(0)$ and $p_0 := \lim_{t\downarrow 0} p(t)$, where the convergence of the improper integral as $t \downarrow 0$ will be verified in Lemma \ref{lem:h-p}(ii).

Let $\gamma$ denote the second gradient component of the objective function.
Then it needs to satisfy the relations $\gamma(p(t))=h(t)$ for $t\in[0,T]$ and $\gamma(\eta(t))=\eta(t)$ for $t>T$.
As we prove in Lemma~\ref{lem:h-p}(i), $p$ is strictly decreasing on $[0,T]$, so the relation $\gamma(p(t))=h(t)$ determines $\gamma=h\circ p^{-1}$ on $[\eta(T),p_0]$.
For $q\le\eta(T)$, we require $\gamma(q)=q$ so that the second gradient component agrees with that of the quadratic objective when $t\ge T$.
Above $p_0$, we need a smooth extension whose derivative has a positive lower bound and a finite upper bound.
Define the transition function $\zeta:\R\to[0,1]$ and the function $a_\gamma:[p_0,+\infty)\to(0,+\infty)$ by
\begin{equation}
  \zeta(q):=\vartheta(3(q-p_0)-1)
  \qquad\mbox{and}\qquad
  a_\gamma(q):=
  \frac{24(1-\zeta(q))}{\sqrt{9h_0^2+96(q-p_0)}}+\zeta(q).
  \label{eq:zeta-definition}
\end{equation}
The function $a_\gamma(q)$ is designed to match the derivative of a smooth local extension of $h\circ p^{-1}$ near $p_0$, while it equals $1$ for $q\ge p_0+2/3$.
Using the function $a_\gamma$ in \eqref{eq:zeta-definition} to extend $\gamma$ to $q\ge p_0$, define $\gamma:\R\to\R$ by
\begin{equation}
  \gamma(q):=
  \begin{cases}
    q,&q\le\eta(T), \\[1mm]
    h(p^{-1}(q)),&\eta(T)<q<p_0, \\[1mm]
    h_0+\displaystyle\int_{p_0}^q a_\gamma(r)\dd r,&q\ge p_0.
  \end{cases}
  \label{eq:gamma-definition}
\end{equation}
Integrating $\gamma$, we define the objective $f_{\rm c}:\R^2\to\R$ and its initial point by
\begin{equation}
\label{eq:continuous-objective}
    \begin{aligned}
        f_{\rm c}(x,y)&:=\frac12x^2+\int_0^y\gamma(q)\dd q
  \\
  \mbox{and}\quad x_0&:=\left(\frac12,p_0\right).
    \end{aligned}
\end{equation}

Before proceeding to the verification of the counterexample in the next subsection, we illustrate in Figure~\ref{fig:continuous-counterexample} the restart quantity and the objective gap along the constructed trajectory.

\begin{figure}[!ht]
  \centering
  \includegraphics[width=0.8\textwidth]{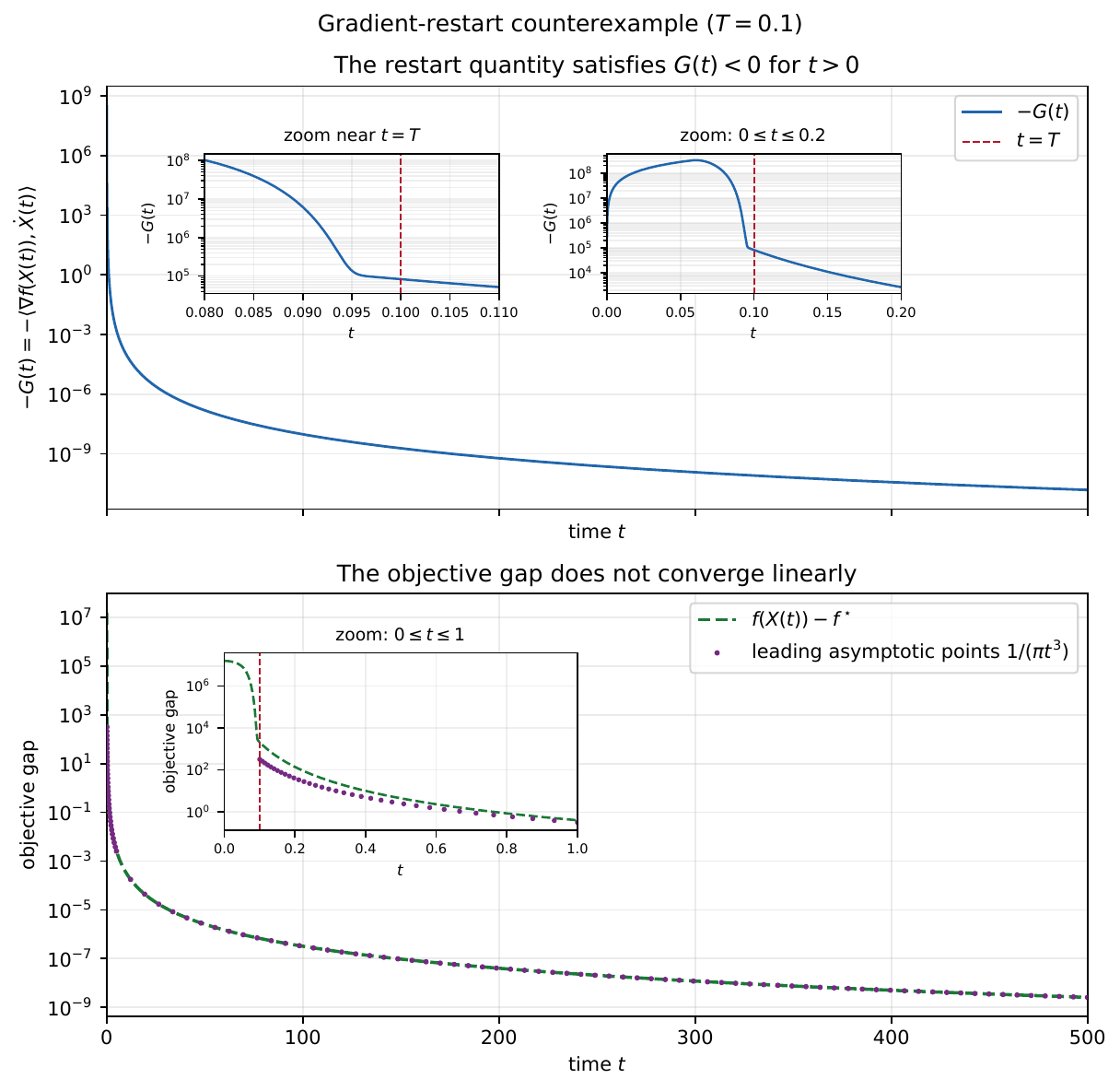}
\caption{Numerical evaluation of the continuous construction formulas with $T=0.1$.
The upper panel plots $-G(t)>0$, where $G$ is the restart quantity in \eqref{eq:continuous-gr}, and the lower panel compares the objective gap with its leading asymptotic $1/(\pi t^3)$.}
  \label{fig:continuous-counterexample}
\end{figure}

\subsection{Verification of the continuous counterexample}
\label{subsec:continuous-prove}

In this subsection, we prove Theorem~\ref{thm:continuous-counterexample} and Corollary~\ref{coro:continuous-counterexample} for the objective $f_{\mathrm{c}}$ and initial point $x_0$ in \eqref{eq:continuous-objective}.
To describe the trajectory generated by the NAG-ODE \eqref{eq:nag-ode}, we define the concatenated curve $X_{\mathrm{c}} : [0, +\infty) \to \mathbb{R}^2$ by
\begin{equation}
\label{eq:continuous-trajectory}
X_{\mathrm{c}}(t) := \begin{cases}
(u(t), p(t)), & 0 \le t < T, \\
(u(t), \eta(t)), & t \ge T.
\end{cases}
\end{equation}
In Lemmas~\ref{lem:quantitative-bessel}--\ref{lem:gamma} below, we successively establish the quantitative properties of $(u, \eta)$ and $(h, p)$, prove that $\gamma \in C^\infty(\mathbb{R})$, and show that $X_{\mathrm{c}}$ solves the NAG-ODE \eqref{eq:nag-ode} for $f_{\mathrm{c}}$ with $X_{\mathrm{c}}(0) = x_0$ and $\dot{X}_{\mathrm{c}}(0) = 0$.
We then use these results to verify that $f_{\mathrm{c}} \in C^\infty(\mathbb{R}^2) \cap \mathcal{S}_{\mu,L}(\mathbb{R}^2)$, prove that $G(t) < 0$ for all $t > 0$, and determine the asymptotic objective gap.

We first record several quantitative bounds on the scaled Bessel functions $u$ and $\eta$.
\begin{lemma}
\label{lem:quantitative-bessel}
Let $T:=1/10$.
The scaled Bessel functions $u$ and $\eta$ defined in \eqref{eq:bessel-branches} satisfy
\begin{align}
  &0<u(t)<\frac12,\quad u'(t)<0,\quad
   \eta(t)>0,\quad\mbox{and}\quad -\frac2{t^3}<\eta'(t)<0
  &&\quad\text{for }0<t\le T,
  \label{eq:small-time-signs}
  \\
  &\mbox{and}\quad\eta(t)\le\eta(T)
  &&\quad\text{for }t\ge T.
  \label{eq:tail-height-bound}
\end{align}
Moreover, the constant $N_T$ defined in \eqref{eq:NT-definition} is positive and
\begin{equation}
  64<\eta(T)<65
  \qquad\mbox{and}\qquad
  1276<-\eta'(T)<1277.
  \label{eq:matching-time-bessel-bounds}
\end{equation}
\end{lemma}

The proof of Lemma~\ref{lem:quantitative-bessel} is deferred to Appendix~\ref{app:continuous-matching-time}.
By Lemma~\ref{lem:quantitative-bessel}, specifically \eqref{eq:small-time-signs}, we have $u(t)u'(t) < 0$ for $t \in (0, T]$.
For $t\ge T$, we also have $\eta(t)\le\eta(T)$ by \eqref{eq:tail-height-bound} in Lemma~\ref{lem:quantitative-bessel}, so the trajectory $X_{\rm c}$ defined in \eqref{eq:continuous-trajectory} remains in the quadratic half-space $y\le\eta(T)$.

A smooth function $v$ is said to be flat at an endpoint $a$ if $v^{(k)}(a)=0$ for every integer $k\ge0$, with derivatives taken from within its domain.
We next prove that the functions $h$ and $p$ in \eqref{eq:h-definition} and \eqref{eq:p-definition} are smooth on $[0,T]$, with $p'(0)=0$, and match $\eta$ to all orders at $T$.

\begin{lemma}
\label{lem:h-p}
Let the functions $h,p:[0,T]\to\R$ be defined by \eqref{eq:h-definition} and \eqref{eq:p-definition}.
Then $h,p\in C^\infty([0,T])$, with the following properties.
\begin{enumerate}
\item[(i)] For $0<t\le T$, we have $h(t)>0$, $h'(t)<0$, and $p'(t)<0$.
On this interval, the function $p$ solves the differential equation
\begin{equation}
\label{eq:p-equation}
  p''(t)+\frac3t p'(t)+h(t)=0.
\end{equation}

\item[(ii)] On $[0,T/2]$, the functions $h$ and $p$ have the polynomial expressions
\begin{equation}
\label{eq:h-origin}
    \begin{aligned}
        h(t)&=h_0-t^2\\
  \mbox{and}\quad p(t)&=p_0-\frac{h_0}{8}t^2+\frac1{24}t^4
    \end{aligned}
\end{equation}
with $h_0=h(0)>0$ and $p_0=p(0)=\lim_{t\downarrow0}p(t)>\eta(T)$.
\item[(iii)] The functions $h-\eta$ and $p-\eta$ are flat at $T$ from the left.
In particular,
\begin{equation}
  h(T)=\eta(T),
  \qquad
  p(T)=\eta(T),
  \qquad\mbox{and}\qquad p'(T)=\eta'(T).
  \label{eq:p-terminal-data}
\end{equation}
\end{enumerate}
\end{lemma}

\begin{proof}
For (i), the transition function $\theta_T$ in \eqref{eq:smooth-step} is smooth and nondecreasing, with $\theta_T$ flat at $T/2$ and $1-\theta_T$ flat at $T$.
Consequently, the functions $g_{\rm b}$ and $\psi_T$ defined in \eqref{eq:continuous-baseline} and \eqref{eq:NT-definition} are smooth on $[0,T]$, with $\psi_T\ge0$.
By Lemma~\ref{lem:quantitative-bessel}, we have $-\eta'>0$ on $(0,T]$ and $N_T>0$.
The function $g_{\rm b}$ equals $2t>0$ on $(0,T/2]$ and is a convex combination of $2t$ and $-\eta'(t)$ on $(T/2,T]$, so $g_{\rm b}>0$ on $(0,T]$.
Because $N_T\psi_T\ge0$ and $\psi_T=0$ on $[0,T/2]$, the function $g$ in \eqref{eq:g-definition} is smooth on $[0,T]$, positive on $(0,T]$, and equal to $2t$ on $[0,T/2]$.

By its integral definition \eqref{eq:h-definition}, the function $h$ belongs to $C^\infty([0,T])$, with $h'=-g<0$ on $(0,T]$ and $h\ge\eta(T)>0$.
Differentiating the integral in \eqref{eq:p-definition}, we find that $p$ is smooth on $(0,T]$ and that its derivative is given by
\begin{equation}
  p'(t)=-t^{-3}\int_0^t s^3h(s)\dd s<0
  \qquad\text{for }\  0<t\le T.
  \label{eq:p-prime}
\end{equation}
The derivative of $t^3p'(t)$ is therefore $-t^3h(t)$, which is equivalent to \eqref{eq:p-equation} and proves (i).

For (ii), since $h' = -g = -2t$ on $[0, T/2]$, we obtain the expression for $h$ in \eqref{eq:h-origin}, with $h_0 \ge \eta(T) > 0$ by \eqref{eq:h-definition}.
On this interval, the integral in \eqref{eq:p-prime} reduces to
\[
  p'(t)=-t^{-3}\int_0^t s^3(h_0-s^2)\dd s
       =-\frac{h_0}{4}t+\frac16t^3.
\]
Thus, $p(t)+h_0t^2/8-t^4/24$ is constant on $(0,T/2]$, and $p$ has a finite limit as $t\downarrow0$.
This shows the convergence of the improper integral that defines $p(0)$ in \eqref{eq:p-definition} and identifies the integration constant as $p_0$.
The polynomial expression \eqref{eq:h-origin} therefore extends to $t=0$, which proves the smoothness of $p$ on $[0,T]$ and the initial condition $p'(0)=0$.
Since $h>0$, we also have
\[
  p_0-\eta(T)
  =\int_0^T r^{-3}\left(\int_0^r s^3h(s)\dd s\right)\dd r>0.
\]
This proves (ii).

For (iii), note from formula \eqref{eq:continuous-baseline} that $g_{\rm b}+\eta'=(1-\theta_T)(2t+\eta')$ near $T$.
Since $1-\theta_T$ and $\theta_T'$ are flat at $T$, both $g_{\rm b}+\eta'$ and $\psi_T=\theta_T'/I_T$ are flat at $T$.
From the definition of $g$ in \eqref{eq:g-definition} and the flatness of $g_{\rm b}+\eta'$ and $\psi_T$ at $T$, we conclude that $g+\eta'$ is flat at $T$.
From the identity
\[
  h(t)-\eta(t)=\int_t^T\bigl(g(s)+\eta'(s)\bigr)\dd s
\]
and the flatness of $g+\eta'$ at $T$, we conclude that $h-\eta$ is flat at $T$.
Since $\int_0^T s^4 \psi_T(s)\,\mathrm{d}s = 1$, we obtain from the definition of $N_T$ in \eqref{eq:NT-definition} that
\begin{equation}
  \int_0^T s^4g(s)\dd s
  =\int_0^T s^4g_{\rm b}(s)\dd s+N_T=R_T.
  \label{eq:g-moment}
\end{equation}
Noting that $h'=-g$ and $h(T)=\eta(T)$, we have
\begin{equation*}
  \int_0^Tt^3h(t)\dd t
  =\frac14\Big(T^4\eta(T)+\int_0^Ts^4g(s)\dd s\Big)
    =-T^3\eta'(T).
\end{equation*}
Here, the first equality uses integration by parts, while the second one follows from \eqref{eq:g-moment} and the definition of $R_T$ in \eqref{eq:IT-RT}.
Substituting the preceding integral identity into \eqref{eq:p-prime} at $t=T$, we obtain $p'(T)=\eta'(T)$.
The values $h(T)=p(T)=\eta(T)$ follow from \eqref{eq:h-definition} and \eqref{eq:p-definition}, completing the identities in \eqref{eq:p-terminal-data}.

Set $w:=p-\eta$.
Subtracting \eqref{eq:scalar-bessel-equation} for $\eta$ from \eqref{eq:p-equation} and using \eqref{eq:p-terminal-data}, we obtain
\[
  w''(t)+\frac3t w'(t)=-(h(t)-\eta(t))
  \qquad\mbox{and}\qquad w(T)=w'(T)=0.
\]
Because $h-\eta$ is flat at $T$ and $3/t$ is smooth there, we prove inductively by successive differentiation of this equation that every left derivative of $w$ vanishes at $T$, completing (iii).
\end{proof}

To turn the equation for $p$ in Lemma~\ref{lem:h-p}(i) into the second coordinate equation of the NAG-ODE, we next verify $\gamma(p(t))=h(t)$ for $0\le t\le T$ and prove that $\gamma$ is smooth on $\R$.

\begin{lemma}
\label{lem:gamma}
Let $\gamma:\R\to\R$ and $X_{\rm c}:[0,+\infty)\to\R^2$ be defined by \eqref{eq:gamma-definition} and \eqref{eq:continuous-trajectory}.
Then the following statements hold.
\begin{enumerate}
\item[(i)] The function $\gamma$ belongs to $C^\infty(\R)$ and satisfies
\begin{align*}
  \gamma(p(t))&=h(t) &&\text{for }0\le t\le T
  \\
  \mbox{and}\quad \gamma(q)&=q &&\text{for }q\le\eta(T).
\end{align*}

\item[(ii)] The trajectory $X_{\rm c}$ defined by \eqref{eq:continuous-trajectory} solves the NAG-ODE \eqref{eq:nag-ode} for $f_{\rm c}$, with $X_{\rm c}(0)=x_0$ and $\dot X_{\rm c}(0)=0$.
\end{enumerate}
\end{lemma}

\begin{proof}
By Lemma~\ref{lem:h-p}(i), the function $p$ is strictly decreasing on $[0,T]$.
Its endpoint values are $p(0)=p_0$ by Lemma~\ref{lem:h-p}(ii) and $p(T)=\eta(T)$ by Lemma~\ref{lem:h-p}(iii), so it is a bijection from $[0,T]$ onto $[\eta(T),p_0]$.
The inverse $p^{-1}$ is continuous on $[\eta(T),p_0]$ and smooth on $(\eta(T),p_0)$ by the inverse function theorem.
By \eqref{eq:gamma-definition}, $\gamma(q)=h(p^{-1}(q))$ for $\eta(T)<q<p_0$, so $\gamma$ is smooth on $(\eta(T),p_0)$.
Noting from Lemma~\ref{lem:h-p}(i) that $h'<0$ and $p'<0$ on $(0,T)$, we have
\begin{equation}
    \label{eq:p_mono_0T}
    \gamma'(p(t))=\frac{h'(t)}{p'(t)}>0 \qquad \mbox{for all} \qquad 0<t<T.
\end{equation}

At $q=\eta(T)$, extend $h$ and $p$ locally beyond $T$ by $\eta$.
The extensions are smooth by Lemma~\ref{lem:h-p}(iii), and $p'(T)=\eta'(T)<0$.
By the inverse function theorem, the composition $h \circ p^{-1}$ is therefore smooth near $\eta(T)$.
For $q<\eta(T)$ sufficiently close to $\eta(T)$, both extended functions equal $\eta$, so $h(p^{-1}(q))=q$.
Thus the composition agrees with the lower branch of $\gamma$, proving smoothness at $q=\eta(T)$.

At $q=p_0$, the inverse function theorem does not apply directly to $p$, because $p'(0)=0$ by Lemma~\ref{lem:h-p}(ii).
The polynomial formulas in that lemma allow us to use the variable $\tau:=t^2$ instead.
Define the polynomials $H,V:\R\to\R$ by
\[
  H(\tau):=h_0-\tau
  \qquad\mbox{and}\qquad
  V(\tau):=p_0-\frac{h_0}{8}\tau+\frac1{24}\tau^2.
\]
The formulas in Lemma~\ref{lem:h-p}(ii) read $h(t)=H(t^2)$ and $p(t)=V(t^2)$ near zero.
Since $V(0) = p_0$ and $V'(0) = -h_0/8 \neq 0$, the function $V$ admits a smooth local inverse near $q = p_0$ by the inverse function theorem, mapping $p_0$ to $\tau = 0$.
Substituting the local inverse of $V$ into $H$, we obtain the smooth extension $\gamma_{\rm loc}$ of $h\circ p^{-1}$ across $p_0$, with formula and derivative given by
\[
  \gamma_{\rm loc}(q)
  :=\frac{\sqrt{9h_0^2+96(q-p_0)}-h_0}{2}
  \qquad\mbox{and}\qquad
  \gamma_{\rm loc}'(q)
  =\frac{24}{\sqrt{9h_0^2+96(q-p_0)}}>0.
\]
The transition function $\zeta$ in \eqref{eq:zeta-definition} equals zero on $[p_0,p_0+1/3]$ and one on $[p_0+2/3,+\infty)$.
Hence $a_\gamma=\gamma_{\rm loc}'$ on $[p_0,p_0+1/3]$ and $a_\gamma=1$ on $[p_0+2/3,+\infty)$.
Since $\gamma_{\rm loc}(p_0)=h_0$, the integral formula for $\gamma$ in \eqref{eq:gamma-definition} agrees with $\gamma_{\rm loc}$ on $[p_0,p_0+1/3]$.
The function $\gamma$ is therefore smooth at $q=p_0$.
Since $\gamma$ is also smooth at $q=\eta(T)$ and on $\R\setminus\{\eta(T),p_0\}$, we have $\gamma\in C^\infty(\R)$.
For $0<t<T$, we have $\eta(T)<p(t)<p_0$.
Then it follows from definition \eqref{eq:gamma-definition} that $\gamma(p(t))=h(t)$.
At $t=0$ and $t=T$, we obtain from the definition \eqref{eq:gamma-definition} and the endpoint values in Lemma~\ref{lem:h-p}(ii)--(iii) that $\gamma(p_0)=h_0=h(0)$ and $\gamma(\eta(T))=\eta(T)=h(T)$.
Finally, $\gamma(q)=q$ for $q\le\eta(T)$ by \eqref{eq:gamma-definition}, completing (i).

For (ii), the gradient of the objective in \eqref{eq:continuous-objective} is $\nabla f_{\rm c}(x,y)=(x,\gamma(y))$, and the first coordinate $u$ of $X_{\rm c}$ satisfies \eqref{eq:scalar-bessel-equation}.
For the second coordinate on $(0,T)$, substitute $\gamma(p(t))=h(t)$ from (i) into the differential equation for $p$ established in Lemma~\ref{lem:h-p}(i).
Thus $p$ satisfies the equation for the second coordinate of the NAG-ODE for $f_{\rm c}$:
\[
  p''(t)+\frac3t p'(t)+\gamma(p(t))=0.
\]
For $t\ge T$, we have $\eta(t)\le\eta(T)$ by Lemma~\ref{lem:quantitative-bessel}.
Thus $\gamma(\eta(t))=\eta(t)$ by (i), and \eqref{eq:scalar-bessel-equation} is the NAG-ODE equation for the second coordinate $\eta$ on $[T,+\infty)$.
Since $p - \eta$ is flat at $T$ by Lemma~\ref{lem:h-p}(iii) and the first coordinate is $u$ throughout $[0, \infty)$, the concatenated curve $X_{\mathrm{c}}$ belongs to $C^\infty([0, \infty); \R^2)$.
Combining the equations on $(0, T)$ and $[T, \infty)$, we conclude that $X_{\mathrm{c}}$ solves the NAG-ODE~\eqref{eq:nag-ode} for all $t > 0$.
Finally, combining $u(0) = 1/2$ and $u'(0) = 0$ with the polynomial expression for $p$ in Lemma~\ref{lem:h-p}(ii), we obtain $X_{\mathrm{c}}(0) = x_0$ and $\dot{X}_{\mathrm{c}}(0) = 0$.
This proves (ii).
\end{proof}

To complete the proof of Theorem~\ref{thm:continuous-counterexample}, it remains to prove that the objective $f_{\rm c}$ in \eqref{eq:continuous-objective} is strongly convex with a Lipschitz-continuous gradient and that the restart quantity $G$ in \eqref{eq:continuous-gr} satisfies $G(t)<0$ for every $t>0$.

\begin{proof}[Proof of Theorem~\ref{thm:continuous-counterexample}]
Take the objective $f:=f_{\rm c}$ and the initial point $x_0$ as defined in \eqref{eq:continuous-objective}.
By Lemma~\ref{lem:gamma}(i), we have $\gamma \in C^\infty(\mathbb{R})$ and hence $f \in C^\infty(\mathbb{R}^2)$, with gradient and Hessian given by
\[
  \nabla f(x,y)=\begin{pmatrix}x\\\gamma(y)\end{pmatrix}
  \qquad\mbox{and}\qquad
  \nabla^2 f(x,y)=
  \begin{pmatrix}1&0\\0&\gamma'(y)\end{pmatrix}.
\]
By \eqref{eq:p_mono_0T}, established in the proof of Lemma~\ref{lem:gamma}, we have $\gamma'(p(t))>0$ for $0<t<T$, and hence $\gamma'>0$ on $(\eta(T),p_0)$.
On the remaining intervals, we obtain from the definition of $\gamma$ in \eqref{eq:gamma-definition} that $\gamma'(q) = 1$ for $q \le \eta(T)$ and $\gamma'(q) = a_\gamma(q)$ for $q \ge p_0$.
Note from the proof of Lemma~\ref{lem:gamma} that $a_\gamma(q)>0$ for all $q\ge p_0$ and $a_\gamma(q)=1$ on $[p_0+2/3,+\infty)$.
Since $\gamma'$ is positive and continuous on the compact interval $[\eta(T), p_0 + 2/3]$ and equals $1$ outside this interval, there exist constants $0 < \underline{m} \le \overline{m} < +\infty$ such that $\underline{m} \le \gamma'(q) \le \overline{m}$ for all $q \in \mathbb{R}$.
The eigenvalues of $\nabla^2 f$ therefore lie between $\mu$ and $L$, where $\mu:=\min\{1,\underline m\}$ and $L:=\max\{1,\overline m\}$.
This proves that $f\in C^\infty(\R^2)\cap\mathcal S_{\mu,L}(\R^2)$.
Since $\eta(T)>0$ and $\gamma(q)=q$ for $q\le\eta(T)$, the integral in \eqref{eq:continuous-objective} can be evaluated explicitly to obtain
\begin{equation}
  f_{\rm c}(x,y)=\frac12(x^2+y^2)
  \qquad \text{when }y\le\eta(T).
  \label{eq:continuous-quadratic-half-space}
\end{equation}
In particular, $\nabla f(0,0)=0$ and $f(0,0)=0$, so the unique minimizer is $x^\star=(0,0)$ and $\fstar=0$.
By Lemma~\ref{lem:gamma}(ii), $X:=X_{\rm c}$ solves the NAG-ODE \eqref{eq:nag-ode} with the prescribed initial conditions.
Because $\nabla f$ is globally Lipschitz continuous, this solution is unique by \cite[Theorem~1]{su2016}.

Finally, it suffices to verify \eqref{eq:continuous-main-nonrestart}, i.e., $G(t)<0$ for all $t>0$.
On $(0,T]$, we have $u>0$ and $u'<0$ by Lemma~\ref{lem:quantitative-bessel}, while $h>0$ and $p'<0$ by Lemma~\ref{lem:h-p}(i).
Since $\gamma(p(t))=h(t)$ for $0\le t\le T$ by Lemma~\ref{lem:gamma}(i), the restart quantity $G$ in \eqref{eq:continuous-gr}, i.e., $G(t)=\ip{\nabla f(X(t))}{\dot X(t)}$ for $t>0$, satisfies
\[
  G(t)=u(t)u'(t)+h(t)p'(t)<0
  \qquad\text{for }0<t\le T,
\]
so no restart occurs on this interval.
For $t \ge T$, we have $\eta(t) \le \eta(T)$ by Lemma~\ref{lem:quantitative-bessel}.
Thus $X(t)$ lies in the half-space where $f$ has the quadratic form \eqref{eq:continuous-quadratic-half-space}, and
\[
  G(t)=u(t)u'(t)+\eta(t)\eta'(t)
      =\frac12\frac{\mathrm d}{\mathrm dt}
        \left(\frac{A(t)}{t^2}\right)
  \qquad\mbox{and}\qquad A(t):=J_1(t)^2+Y_1(t)^2.
\]
Since the integrand $\ee^{-v\cosh s}$ in \eqref{eq:K0-integral} is positive and nonincreasing in $v$, the function $K_0$ is positive and nonincreasing on $(0,+\infty)$.
Nicholson's formula in Lemma~\ref{lem:bessel-nicholson}, with $m=1$, expresses $A$ as the integral of $K_0(2t\sinh r)$ against the positive weight $8\cosh(2r)/\pi^2$.
Consequently, $A$ is positive and nonincreasing on $(0,+\infty)$.
Since $A$ is differentiable, $A'\le0$, and hence
\[
  G(t)=\frac{A'(t)}{2t^2}-\frac{A(t)}{t^3}<0
  \qquad\text{for }t\ge T.
\]
This completes the proof.
\end{proof}

\begin{proof}[Proof of Corollary~\ref{coro:continuous-counterexample}]
For the objective $f:=f_{\rm c}$ and the initial point $x_0$ in \eqref{eq:continuous-objective}, we conclude from Theorem~\ref{thm:continuous-counterexample} that no gradient restart occurs, so $X^{\mathrm{gr}}=X_{\rm c}$.
By Lemma~\ref{lem:quantitative-bessel}, we have $\eta(t)\le\eta(T)$ for $t\ge T$, so the trajectory $X_{\rm c}$ defined in \eqref{eq:continuous-trajectory} lies in the half-space $y\le\eta(T)$.
By the quadratic expression \eqref{eq:continuous-quadratic-half-space} and the value $\fstar=0$ established in the proof of Theorem~\ref{thm:continuous-counterexample}, the objective gap is given by
\[
  f(X^{\mathrm{gr}}(t))-\fstar
  =\frac{u(t)^2+\eta(t)^2}{2}
  =\frac{J_1(t)^2+Y_1(t)^2}{2t^2}.
\]
When the expansions \eqref{eq:Jn-large} and \eqref{eq:Yn-large} in Lemma~\ref{lem:bessel-asymptotics} with $m=1$ are squared and added, the mixed sine-cosine terms cancel.
Hence we have $J_1(t)^2 + Y_1(t)^2 = 2/(\pi t) + O(t^{-3})$ as $t \to +\infty$, from which we obtain \eqref{eq:continuous-main-properties}.
For every $C,c_0>0$, the ratio of $f(X^{\mathrm{gr}}(t))-\fstar$ to $C\ee^{-c_0t}$ tends to infinity as $t\to+\infty$, so no such exponential upper bound can hold for all sufficiently large times.
\end{proof}

\begin{remark}
\label{rem:small-matching-times}
The choice $T = 1/10$ is not essential.
Indeed, the proofs of Lemmas~\ref{lem:h-p} and \ref{lem:gamma} and Theorem~\ref{thm:continuous-counterexample} depend on $T$ only through the sign conditions \eqref{eq:small-time-signs}, the tail bound \eqref{eq:tail-height-bound}, and the positivity of $N_T$ in Lemma~\ref{lem:quantitative-bessel}.
Using Lemmas~\ref{lem:bessel-differentiation} and \ref{lem:bessel-ascending}, we have $u(t) \to 1/2$, $u'(t) \sim -t/8$, $\eta(t) \sim 2/(\pi t^2)$, and $\eta'(t) \sim -4/(\pi t^3)$ as $t \downarrow 0$, so there exists $\delta > 0$ such that \eqref{eq:small-time-signs} holds for every $T \in (0, \delta]$.
Moreover, since $\eta(t) \to 0$ as $t \to +\infty$ by \eqref{eq:Yn-large} in Lemma~\ref{lem:bessel-asymptotics}, $\eta$ is bounded above on $[\delta, +\infty)$.
Because $\eta(T) \to +\infty$ as $T \downarrow 0$ and $\eta$ is strictly decreasing on $(0, \delta]$, we obtain $\eta(t) \le \eta(T)$ for all $t \ge T$ whenever $T \in (0, \delta)$ is sufficiently small, which verifies \eqref{eq:tail-height-bound}.
Substituting the small-argument asymptotics of $\eta$ and $\eta'$ into \eqref{eq:continuous-baseline}--\eqref{eq:NT-definition}, we further obtain $\int_0^T s^4 g_{\mathrm{b}}(s)\,\mathrm{d}s = O(T^2)$ and $R_T \to 16/\pi$, and hence $N_T \to 16/\pi > 0$ as $T \downarrow 0$.
Consequently, for every sufficiently small $T > 0$, all conclusions of Lemmas~\ref{lem:h-p} and \ref{lem:gamma} and Theorem~\ref{thm:continuous-counterexample} remain valid by the same arguments.
Since $\eta(T) \to +\infty$ as $T \downarrow 0$, we can therefore choose the objective $f_{\mathrm{c}}$ to coincide with the isotropic quadratic $(x^2+y^2)/2$ on any prescribed bounded subset of $\mathbb{R}^2$.
\end{remark}

\section{Discussions and conclusions}
\label{sec:conclusion}
We constructed smooth, strongly convex objectives for which gradient restart never activates under the prescribed initializations, for both NAG~\eqref{eq:discrete-nag} and NAG-ODE~\eqref{eq:nag-ode}.
In both examples, the restarted and unrestarted methods generate identical trajectories, so gradient restart provides no convergence improvement.
Subsection~\ref{subsec:construction-principles} explains the analytical principles underlying our constructions, while Subsection~\ref{subsec:further-questions} discusses natural extensions and open questions.

\subsection{Motivation and principles of the constructions}
\label{subsec:construction-principles}

Our counterexamples start from the mechanism by which oscillations force gradient restarts on quadratic objectives~\cite{bao2026}.
For the one-dimensional quadratic $z \mapsto z^2/2$ with a nonzero initial point $z_0$, the unique solution of the NAG-ODE~\eqref{eq:nag-ode} starting from rest is $2z_0 u(t)$, which involves only the first-kind scaled Bessel function $u = J_1/t$ in~\eqref{eq:bessel-branches}.
By Lemmas~\ref{lem:bessel-differentiation} and~\ref{lem:bessel-asymptotics}, $u(t) = \sqrt{2/\pi}\, t^{-3/2}\cos(t - 3\pi/4) + O(t^{-5/2})$ and $u'(t) = -J_2(t)/t$, so the restart quantity $G(t)=4z_0^2 u(t)u'(t)$ has a leading term proportional to $-t^{-3}\cos(2t)$ that periodically changes sign.
To suppress this oscillation in the radial norm, we move to $\R^2$ and pair the first-kind solution with its second-kind counterpart ($\eta = -Y_1/t$ in continuous time and $Q_m^{(-1)}$ in discrete time), whose leading asymptotic oscillation replaces $\cos$ with $\sin$ (Lemmas~\ref{lem:ferrers-asymptotics} and~\ref{lem:bessel-asymptotics}).
Geometrically, combining these two out-of-phase branches turns the one-dimensional back-and-forth oscillation into a two-dimensional spiral whose leading $\cos^2$ and $\sin^2$ terms sum to one, so we obtain the non-oscillatory $1/(\pi t^3)$ leading term of the continuous-time objective gap in Corollary~\ref{coro:continuous-counterexample}.
Beyond the leading asymptotics, we conclude from the Nicholson-type integral formula in Lemma~\ref{lem:bessel-nicholson} that the scaled Bessel functions $u$ and $\eta$ defined in \eqref{eq:bessel-branches} satisfy $u(t)u'(t)+\eta(t)\eta'(t)<0$ for every $t>0$, so the restart criterion is never triggered along the curve $(u,\eta)$ for the quadratic objective $(x,y)\mapsto(x^2+y^2)/2$.

The essential distinction in our counterexamples lies between starting on a quadratic objective with zero momentum and entering a quadratic region with a prescribed state.
On a globally quadratic objective, zero-momentum initialization cannot excite the second-kind branch: in continuous time, $\eta$ diverges at the origin, whereas in discrete time, $y_1 = x_1 = Z_0$ would force $x_2 = aZ_0$ rather than $Z_1 = aZ_0 - (0, d)$ with $d > 0$.
To avoid initializing at rest inside the quadratic domain, we confine the isotropic quadratic objective to a half-space containing the origin, which the trajectory enters only after a prescribed initial phase.
Outside this half-space, we treat the objective modification as undetermined and work backward from the necessary conditions for the main theorems to hold.
In discrete time, matching the first two iterates with $\{Z_k\}$ directly determines the required gradients at the initial points via the NAG update, from which candidate objectives are obtained by simple interpolation.
In continuous time, the connecting trajectory $p$ and the second gradient component $\gamma$ are both unknown and coupled through the NAG-ODE via $p'' + 3p'/t + h = 0$ with $h(t) = \gamma(p(t))$.
We therefore transfer the necessary conditions onto the auxiliary function $g = -h'$, construct $g$ explicitly, and successively recover $(h, p)$, $\gamma$, and $f_{\mathrm{c}}$.
In both settings, the trajectory enters the quadratic half-space with nonzero momentum that activates both fundamental branches, rendering the zero-velocity quadratic restart guarantee of \cite{bao2026} inapplicable. 
\subsection{Further Discussion and Extensions}
\label{subsec:further-questions}

The Ferrers and Bessel constructions suggest a connection between the discrete and continuous counterexamples.
By Lemmas~\ref{lem:ferrers-asymptotics} and \ref{lem:discrete-orbit}, the discrete objective gap satisfies $f_{\mathrm{d}}(Z_k) - f_{\mathrm{d}}(0) = O(\rho^{2k}k^{-3})$ as $k \to +\infty$, whereas the continuous objective gap is asymptotic to $1/(\pi t^3)$ by Corollary~\ref{coro:continuous-counterexample}.
For a fixed quadratic curvature $\mu_{\mathrm{d}}$ and a step size $0 < s < 1/\mu_{\mathrm{d}}$, the corresponding Ferrers parameter $\rho_s := \sqrt{1 - s\mu_{\mathrm{d}}}$ approaches one as $s \downarrow 0$.
At a heuristic level, as $\rho_s \to 1$, the exponential factor $\rho_s^{2k}$ in the discrete bound formally disappears and leaves the algebraic $k^{-3}$ decay, which mirrors the sublinear $O(t^{-3})$ rate in continuous time and illustrates the structural parallel between the two constructions.

The explicit constructions in this paper admit several extensions.
First, by Remark~\ref{rem:discrete-extensions}, we can extend the step size of the discrete construction to the endpoint $s=1/L$ and can also choose the objective to be $C^\infty$.
Moreover, by Remark~\ref{rem:small-matching-times}, we can choose the objective in the continuous construction to coincide with $\norm{z}^2/2$ on any prescribed bounded set.
Furthermore, although both counterexamples are two-dimensional, they can be extended to any dimension $n\ge2$ by adding a suitable strongly convex quadratic term in the additional coordinates and initializing those coordinates at zero.

Although these extensions show that non-activation can occur across a broad range of settings, the consistent empirical effectiveness of gradient restart in numerical studies \cite{odonoghue2015,su2016} suggests that such worst-case counterexamples do not reflect its typical behavior. This contrast naturally motivates studying the convergence of gradient restart for generic initializations. Specifically, for a fixed $f\in\mathcal{S}_{\mu,L}(\mathbb{R}^n)$, a central question is whether the gradient-restarted NAG-ODE \eqref{eq:nag-ode} admits global or local linear convergence guarantees for Lebesgue-almost every initial point. Likewise, for NAG \eqref{eq:discrete-nag} with a fixed step size $0<s<1/L$, we ask whether gradient restart improves the linear convergence rate relative to the unrestarted method for Lebesgue-almost every initial point, apart from finite termination. Resolving these questions would help bridge the gap between the practical performance of gradient restart and its theoretical convergence guarantees.

\section*{Acknowledgments}

GPT-6 was used as an auxiliary tool to help develop the counterexample construction and check spelling and grammar.
All mathematical arguments were independently completed by the authors.

\appendix
\small 
\section{Proofs of auxiliary lemmas}

\subsection{Proof of Lemma~\ref{lem:discrete-parameters}}
\label{app:discrete-parameter-check}

We verify the bounds in Lemma~\ref{lem:discrete-parameters} that make $h_{\rm d}$ in \eqref{eq:discrete-h} well defined and place the extrapolated point $y_2$ in the half-space $v\le c$.

\begin{proof}
Recall the constants in \eqref{eq:discrete-parameters}--\eqref{eq:U-d}.
Since $18/19<\rho<19/20$ and $3<\pi<22/7<\sqrt{10}$,
\[
  U>\frac\rho\pi>\frac{18}{19}\frac7{22}
  >\frac{127}{1000}=c+2\varepsilon.
\]
Since $a=\rho^2>0$, retaining the first five terms of the power series for $\operatorname{arctanh}\rho/\rho$, we have
\[
  S:=\frac{\operatorname{arctanh}\rho}{\rho}
  =\sum_{j=0}^{+\infty}\frac{a^j}{2j+1}
  >1+\frac a3+\frac{a^2}{5}+\frac{a^3}{7}+\frac{a^4}{9}
  >\frac{163}{100}.
\]
Using the formulas for $U$ and $d$ in \eqref{eq:U-d} and the lower bound for $S$ just obtained, we estimate $aU-d$ as
\begin{align*}
aU-d
  =\frac{\rho}{\pi}
    \left(\frac7{30}+\frac9{100}S\right)
  >\frac{63}{209}
    \left(\frac7{30}+\frac9{100}\frac{163}{100}\right)
   >\frac{567}{5000}
   =a(c+\varepsilon).
\end{align*}
Since $d=2\rho/(3\pi)>0$, we have proved $0<d<a(U-c-\varepsilon)$.

Using the formulas \eqref{eq:tail-first-two} from Lemma~\ref{lem:ferrers-tail}(i) for the first two terms $Z_0$ and $Z_1$ of the sequence in \eqref{eq:discrete-tail}, we write
\[
  \frac{5Z_1-Z_0}{4}
  =\left(\frac7{160},b\right)
  \qquad\mbox{and}\qquad
  b:=\frac78U-\frac54d.
\]
In particular, $b>0$ because $U>\rho/\pi>(10/7)d$.
Since $\rho<19/20$, we have $\operatorname{arctanh}\rho<\frac12\log 39<2$.
From the definition \eqref{eq:U-d}, the bound $\operatorname{arctanh}\rho<2$, and the inequalities $\rho<1$ and $\pi>3$, we obtain $U<2/5$.
Since $\pi<\sqrt{10}$, we also obtain $d>1/5$ from the formula $d=2/(\pi\sqrt{10})$ in \eqref{eq:U-d}.
Consequently, $b<1/10$ and
\[
  \frac{\norm{5Z_1-Z_0}^2}{16}
  <\frac{49}{25600}+\frac1{100}
  <\frac1{64}=c^2.
\]
\end{proof}

\subsection{Proof of Lemma~\ref{lem:quantitative-bessel}}
\label{app:continuous-matching-time}

To prove Lemma~\ref{lem:quantitative-bessel}, we first bound $J_1$, $J_2$, $Y_1$, and $Y_2$ on $(0,2T]$.
Using these bounds, we verify the signs in \eqref{eq:small-time-signs} and the positivity of the constant $N_T$ defined in \eqref{eq:NT-definition}.
We then combine the decrease of $\eta$ on $[T,2T]$ with Nicholson's formula to prove $\eta(t)\le\eta(T)$ for every $t\ge T$.

\begin{proof}
Let $T:=1/10$ and put $x:=t^2/4$.
For $m=1,2$, set
\[
  S_m(x):=\sum_{j=0}^{+\infty}
  \frac{(-1)^j(H_j+H_{m+j}-2\gamma_{\rm E})}
       {j!(m+j)!}x^j.
\]
Here $\gamma_{\rm E}$ and $H_m$ are Euler's constant and the harmonic numbers introduced in Lemma~\ref{lem:bessel-ascending}, i.e., $H_0=0$ and $H_m=\sum_{\ell=1}^m\ell^{-1}$ for $m\ge1$.
The series for $Y_m$ in Lemma~\ref{lem:bessel-ascending}, given in \eqref{eq:Yn-ascending}, can then be written as
\begin{equation}
  Y_m(t)
  =-\frac1\pi\left(\frac t2\right)^{-m}
    \sum_{j=0}^{m-1}\frac{(m-j-1)!}{j!}x^j
   +\frac2\pi\log\left(\frac t2\right)J_m(t)
   -\frac1\pi\left(\frac t2\right)^mS_m(x).
  \label{eq:Yn-complete-small}
\end{equation}

For $0<t\le2T$, we have $x\le10^{-2}$, so the ratio of successive absolute terms in the series for $J_m$ in \eqref{eq:J-definition} is at most $1/200$ for $m=1,2$.
For each of $J_1$ and $J_2$, the first partial sum is an upper bound and the second partial sum is a positive lower bound, so
\begin{equation}
  0<J_1(t)<\frac t2
  \qquad\mbox{and}\qquad
  0<J_2(t)<\frac{t^2}{8}.
  \label{eq:J12-explicit-bounds}
\end{equation}
For the remaining series $S_m$, the bound $0.577<\gamma_{\rm E}<0.578$ from \cite[Eq.~(5.2.3)]{olver2010nist} and the inequality $H_m\le m$ control each coefficient after the constant term.
In particular,
\[
  \left|S_m(x)-\frac{H_m-2\gamma_{\rm E}}{m!}\right|
  \le\sum_{j=1}^{+\infty}(2j+m+2)x^j
  =\frac{2x}{(1-x)^2}+\frac{(m+2)x}{1-x}.
\]
For $m=1$, the constant term lies in $(-0.156,-0.154)$ and the absolute value of the remainder is less than $0.051$.
For $m=2$, the constant term lies in $(0.172,0.173)$ and the corresponding remainder has absolute value less than $0.061$.
Hence
\begin{equation}
  -1<S_1(x)<0
  \qquad\mbox{and}\qquad
  0<S_2(x)<1.
  \label{eq:S12-explicit-bounds}
\end{equation}
Since $\log(t/2)<0$ on $0<t\le2T=1/5$ and $J_m(t)>0$ for $m=1,2$ by \eqref{eq:J12-explicit-bounds}, the logarithmic term in \eqref{eq:Yn-complete-small} is negative for both $m=1$ and $m=2$.
Using the bounds on $S_1$ and $S_2$ in \eqref{eq:S12-explicit-bounds}, we then obtain from \eqref{eq:Yn-complete-small} that
\[
  Y_1(t)<-\frac2{\pi t}+\frac{t}{2\pi}<0
  \qquad\mbox{and}\qquad Y_2(t)<0,
\]
and, since $\eta'=Y_2/t$ by \eqref{eq:bessel-scaled-derivative} in Lemma~\ref{lem:bessel-differentiation},
\[
  0<t^3\bigl(-\eta'(t)\bigr)
  <\frac4\pi+\frac{t^2}\pi
   +\frac{t^4}{4\pi}\left(\log\frac2t+1\right)<2.
\]
Both $t^2$ and $t^4(\log(2/t)+1)$ are increasing on $(0,1/5]$, so we obtain an upper bound strictly below $2$ by evaluating the displayed expression at $t=1/5$ and using $\log 10<3$ and $\pi>3.141$.
Recall the scaled Bessel functions $u=J_1/t$ and $\eta=-Y_1/t$ defined in \eqref{eq:bessel-branches}.
By Lemma~\ref{lem:bessel-differentiation} with $m=1$, we also have $u'=-J_2/t$.
From the bounds on $J_1$, $J_2$, $Y_1$, and $Y_2$, we obtain \eqref{eq:small-time-signs}, in fact for $0<t\le2T$.

At $t = T$, we obtain $0.049 < J_1(T) < 0.05$ and $0 < J_2(T) < 0.00125$ from the alternating series for $J_1$ and $J_2$ in Lemma~\ref{lem:bessel-ascending}.
Substituting these bounds and \eqref{eq:S12-explicit-bounds} into \eqref{eq:Yn-complete-small}, with $3.141<\pi<3.142$ and $2.9<\log 20<3$, we obtain the endpoint estimates
\[
  64.39<\eta(T)<64.63
  \qquad\mbox{and}\qquad
  1276.25<-\eta'(T)<1276.70.
\]
Thus \eqref{eq:matching-time-bessel-bounds} holds.
From the bounds on $\eta(T)$ and $-\eta'(T)$ just proved and the definition of $R_T$ in \eqref{eq:IT-RT}, we have $R_T > 4T^3(1276) - T^4(65) > 5.097$.
The function $g_{\rm b}$ in \eqref{eq:continuous-baseline} equals $2t$ on $[0,T/2]$ and is a convex combination of $2t$ and $-\eta'$ on $[T/2,T]$.
Since $0<-\eta'(t)<2/t^3$ on $[T/2,T]$,
\[
  0\le\int_0^T s^4g_{\rm b}(s)\dd s
  \le2\int_0^T s^5\dd s+2\int_{T/2}^T s\dd s<0.008,
\]
so $N_T>5.097-0.008>0$ by its definition in \eqref{eq:NT-definition}.

Finally, $\eta'<0$ on $[T,2T]$, so $\eta(t)\le\eta(T)$ on this interval.
Recall the function $A:(0,+\infty)\to\R$, i.e., $A(t)=J_1(t)^2+Y_1(t)^2$.
Since $\mathrm{e}^{-v\cosh s} > 0$ is decreasing in $v$, we see from \eqref{eq:K0-integral} that $K_0$ is positive and nonincreasing on $(0, +\infty)$.
Because $\cosh(2r) > 0$, we conclude from Nicholson's formula \eqref{eq:bessel-nicholson} in Lemma~\ref{lem:bessel-nicholson} with $m = 1$ that $A$ is positive and nonincreasing on $(0, +\infty)$.
Consequently, for $t\ge2T$,
\[
  \eta(t)
  \le\frac{\sqrt{A(t)}}t
  \le\frac{\sqrt{A(T)}}{2T}
  =\frac12\sqrt{u(T)^2+\eta(T)^2}
  <\eta(T),
\]
since $u(T)<1/2$ and $\eta(T)>64$.
Combining the decrease of $\eta$ on $[T,2T]$ with the bound $\eta(t)<\eta(T)$ for $t\ge2T$, we obtain \eqref{eq:tail-height-bound}.
\end{proof}

\end{document}